\documentclass[11pt,reqno]
{amsart}
\usepackage{amsmath,epsfig,graphicx,color}
\usepackage{subfig}
\usepackage{dsfont}
\usepackage{ifthen}
\usepackage{amssymb,latexsym}
\usepackage{hyperref}
\hypersetup{
urlcolor=black, 
  menucolor=black, 
  citecolor=black, 
  anchorcolor=black, 
  filecolor=black, 
  linkcolor=black, 
  colorlinks=true,
}
\usepackage{comment}
\usepackage[numbers,sort&compress]{natbib} 
\newcommand{\cip}{\stackrel{\P}{\rightarrow}}

\newcommand{\eid}{\stackrel{d}{=}}

\newcommand{\x}{\mathbf{x}}
\newcommand{\y}{\mathbf{y}}

\usepackage{mathtools}
\mathtoolsset{showonlyrefs}

\definecolor{darkblue}{rgb}{.2, 0.2,.8}
\definecolor{darkgreen}{rgb}{0,0.5,0.3}
\definecolor{darkred}{rgb}{.8, .1,.1}

\newtheorem{lemma}{Lemma}[section]

\newtheorem{theorem}[lemma]{Theorem}

\newtheorem{proposition}[lemma]{Proposition}
\newtheorem{definition}[lemma]{Definition}
\newtheorem{corollary}[lemma]{Corollary}
\newtheorem{example}[lemma]{Example}
\newtheorem{exercise}[lemma]{Exercise}
\newtheorem{remark}[lemma]{Remark}
\newtheorem{fig}[lemma]{Figure}
\newtheorem{tab}[lemma]{Table}

\newcommand{\cid}{\stackrel{d}{\rightarrow}}

\newcommand{\bth}{\begin{theorem}}
\newcommand{\ethe}{\end{theorem}}

\newcommand{\bre}{\begin{remark}\em }
\newcommand{\ere}{\end{remark}}

\newcommand{\ble}{\begin{lemma}}
\newcommand{\ele}{\end{lemma}}

\newcommand{\bde}{\begin{definition}}
\newcommand{\ede}{\end{definition}}
\newcommand{\bco}{\begin{corollary}}
\newcommand{\eco}{\end{corollary}}

\newcommand{\bpr}{\begin{proposition}}
\newcommand{\epr}{\end{proposition}}

\newcommand{\bexer}{\begin{exercise}}
\newcommand{\eexer}{\end{exercise}}

\newcommand{\bexam}{\begin{example}}
\newcommand{\eexam}{\end{example}}

\newcommand{\bfi}{\begin{fig}}
\newcommand{\efi}{\end{fig}}

\newcommand{\btab}{\begin{tab}}
\newcommand{\etab}{\end{tab}}

\newcommand{\Var}{\operatorname{Var}}

\newcommand{\Cov}{\operatorname{Cov}}

\newcommand{\Corr}{\operatorname{Corr}}

\newcommand{\rhs}{right-hand side}
\newcommand{\df}{distribution function}

\newcommand{\beao}{\begin{eqnarray*}}
\newcommand{\eeao}{\end{eqnarray*}\noindent}

\newcommand{\beam}{\begin{eqnarray}}
\newcommand{\eeam}{\end{eqnarray}\noindent}

\newcommand{\beqq}{\begin{equation}}
\newcommand{\eeqq}{\end{equation}\noindent}

\newcommand{\bce}{\begin{center}}
\newcommand{\ece}{\end{center}}

\newcommand{\barr}{\begin{array}}
\newcommand{\earr}{\end{array}}

\newcommand{\vague}{\stackrel{\lower0.2ex\hbox{$\scriptscriptstyle
                    \it{v} $}}{\rightarrow}}
\newcommand{\weak}{\stackrel{\lower0.2ex\hbox{$\scriptscriptstyle
                    \it{w} $}}{\rightarrow}}
\newcommand{\what}{\stackrel{\lower0.2ex\hbox{$\scriptscriptstyle
                    \it{\hat{w}} $}}{\rightarrow}}

\newcommand{\bdis}{\begin{displaymath}}
\newcommand{\edis}{\end{displaymath}\noindent}

\newcommand{\N}{\mathbb{N}}
\newcommand{\R}{\mathbb{R}}

\newcommand{\nto}{n\to\infty}

\newcommand{\ov}{\overline}

\newcommand{\vep}{\varepsilon}

\newcommand{\st}{such that}

\newcommand{\seq}{sequence}

\newcommand{\bfI}{{\bf I}}

\newcommand{\E }{{\mathbb E}}
\renewcommand{\P }{{\mathbb P}}

\newcommand{\1}{\mathds{1}}

\allowdisplaybreaks

\DeclareMathOperator{\e}{e}

\begin{document}
\bibliographystyle{acm}
\title[Interpoint distances]{Maximum interpoint distance of high-dimensional random vectors}
\thanks{Johannes Heiny's and Carolin Kleemann's research was partially supported by grant VR-2023-03577 ``High-dimensional extremes and random matrix structures'' and by  RTG 2131 High-dimensional Phenomena in Probability -- Fluctuations and Discontinuity. We thank Peter Eichelsbacher for fruitful discussions. The authors acknowledge the feedback and constructive comments of 3 anonymous reviewers and an associate editor which led to an improved version of this paper.}

\author[J. Heiny]{Johannes Heiny}
\address{Department of Mathematics,
Stockholm University,
Albano hus 1,
10691 Stockholm,
Sweden}
\email{johannes.heiny@math.su.se}
\author[C. Kleemann]{Carolin Kleemann}
\address{Fakult\"at f\"ur Mathematik,
Ruhruniversit\"at Bochum,
Universit\"atsstrasse 150,
D-44801 Bochum,
Germany}
\email{carolin.kleemann@rub.de}

\begin{abstract}
A limit theorem for the largest interpoint distance of $p$ independent and identically distributed  points in $\R^n$ to the Gumbel distribution is proved, where the number of points $p=p_n$ tends to infinity as the dimension of the points $n\to\infty$. The theorem holds under moment assumptions and corresponding conditions on the growth rate of $p$. We obtain a plethora of ancillary results such as the joint convergence of  maximum and minimum interpoint distances. Using the inherent sum structure of interpoint distances, our result is generalized to maxima of dependent random walks with non-decaying correlations and we also derive point process convergence. An application of the maximum interpoint distance to testing the equality of means for high-dimensional random vectors is presented. Moreover, we study the largest off-diagonal entry of a sample covariance matrix.
The proofs are based on the Chen-Stein Poisson approximation method and Gaussian approximation to large deviation probabilities. 
\end{abstract}
\keywords{Maximum under dependence, high dimension, Gumbel distribution, extreme value theory, $p$-norms,  independence test}
\subjclass{Primary 60G70; Secondary 60G50, 60F10, 60B12}
\maketitle

\section{Introduction}\label{sec:intro}

In this paper we study the asymptotic distribution of the largest interpoint distance
\begin{align*}
	M_{n,p}:=\max_{1\leq i<j\leq p}\Vert\x_i-\x_j\Vert_2\,,
\end{align*}
where $\x_1,\x_2,\ldots, \x_p$ are random vectors in $\R^n$ and $\Vert\cdot\Vert_2$ denotes the Euclidean norm on $\R^n$.  
Interpoint distances are used in a wide range of applications in many areas of probability and statistics, for example
in distributional characterization, classification, independence testing and cluster analysis \cite{li:2018}. Thanks to their simplicity of computation and straightforward geometric interpretation, interpoint distance-based procedures have been particularly appealing to practitioners for analyzing data samples.

Several limit theorems for the largest interpoint distance $M_{n,p}$ of independent and identically distributed (iid) random vectors $\x_1,\ldots,\x_p$ with a fixed dimension $n$ have been proved. 
Typically a distinction is made between bounded and unbounded support of the distribution of the points $\x_i$.
For instance, if the points are uniformly distributed on the two-dimensional unit ball, we can see in Figure~\ref{fig:Bild1} that for a growing number of points, that is $p\to \infty$, the largest interpoint distance $M_{n,p}$ converges to the diameter of the unit ball. Regarding the maximum interpoint distance as the diameter of the convex hull of $p$ independent points, Mayer and Molchanov \cite{mayer:molchanov:2007} and  Lao and Mayer \cite{lao2008Umax} obtained a Weibull distribution as the limiting law of the suitably centered and normalized $M_{n,p}$ in case of points distributed on the $n$-dimensional unit ball (including the uniform distribution). For bounded support and fixed dimension~$n$, Appel et al.\ \cite{appel2002limit} found a limiting distribution for $M_{n,p}$ in the case of uniformly distributed points in a compact set with a well defined major axis and a suitable decay rate at the endpoints.
For a uniform distribution in a proper ellipse with major axis $2$, Jammalamadaka and Janson \cite{jammalamadaka:janson:2015} gave a limiting law for $M_{n,p}$, which involved two independent Poisson processes. This result was generalized by Schrempp \cite{schrempp:2016} to uniform or nonuniform distributions over an $n$-dimensional ellipsoid.

If the support of the $\x_i$ is unbounded, then one single observation/outlier might cause $M_{n,p}$ to be large. An example of a distribution with unbounded support is given in Figure~\ref{fig:Bild2} which shows a cloud of bivariate standard normal distributed points.
In the unbounded case, Matthews and Rukhin \cite{matthews:rukhin:1993} obtained a Gumbel limiting distribution if the points follow a spherical symmetric normal distribution. Henze and Klein \cite{henze:klein:1996}, Jammalamadaka and Janson \cite{jammalamadaka:janson:2015} and Demichel et al.\ \cite{demichel:fermin:soulier:2015} generalized this result to any spherically symmetric distribution.
Complementary to these developments, Jammalamadaka and Janson \cite{jammalamadaka:janson:1986} obtained a limiting distribution for the minimum interpoint distance by considering the asymptotic distribution of a triangular scheme of $U$-statistics. The minimum interpoint distance is usually attained by points in the bulk of the distribution, whereas the maximum interpoint distance is achieved by outliers. Therefore, $M_{n,p}$ is less suitable for goodness of fit tests, but could be used to identify outliers.

\begin{figure}[h] 
	\centering
	\subfloat[10 points]{ \includegraphics[width=0.45\textwidth]{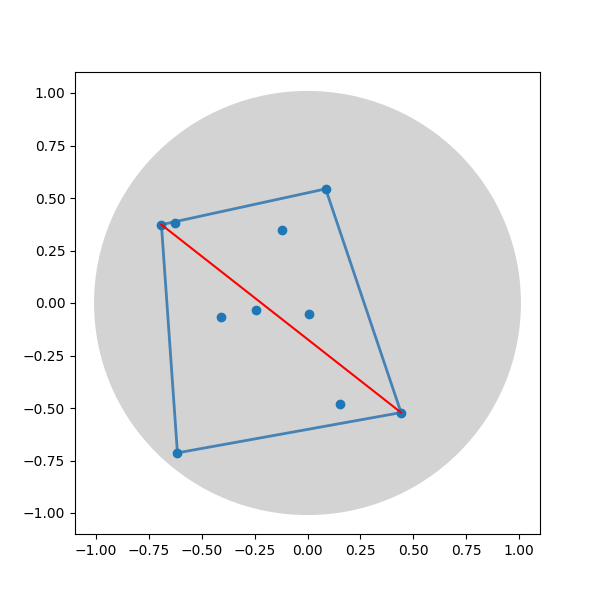}}
	\qquad
	\subfloat[250 points]{ \includegraphics[width=0.45\textwidth]{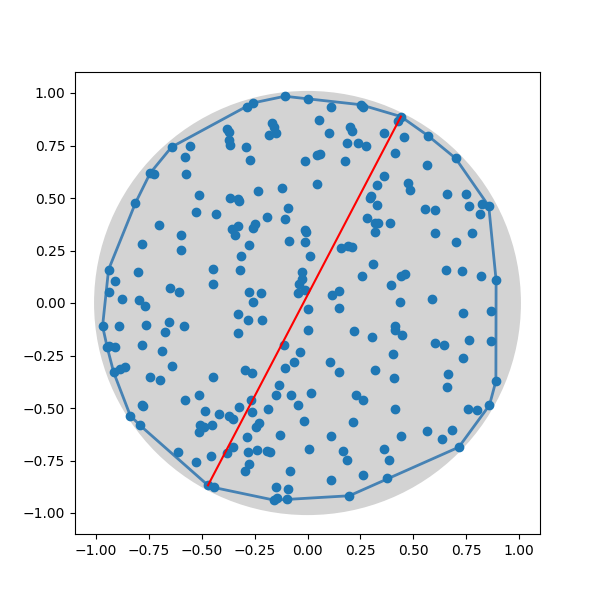}}   
	\caption{Uniformly distributed points on the two-dimensional unit ball}
	\label{fig:Bild1}
\end{figure}

In all of these works, the dimension $n$ is assumed to be fixed. 
Recent technological advances such as the rapid improvement of computing power and measurement devices, however, have greatly facilitated the collection of high-dimensional data.
Huge data sets arise naturally in genome sequence data in biology, online networks, wireless communication, large financial portfolios, and natural sciences. More applications where the dimension $n$ might be of the same or even higher magnitude than the sample size $p$ are discussed in \cite{johnstone:2001,donoho:2000}. In such a high-dimensional setting, one faces new probabilistic and statistical challenges; see \cite{johnstone:titterington:michael:2009} for a review.
Since interpoint distances can be easily computed in any dimension, they provide a promising approach to analyzing high-dimensional data; see \cite{zhu:shao:2021}.

\begin{figure}
	\centering
	\includegraphics[width=0.45\textwidth]{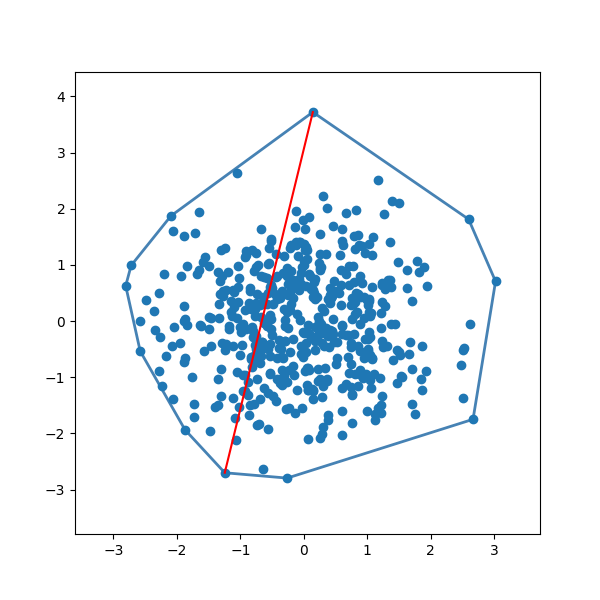}  
	\caption{500 normal distributed points on $\R^2$}
	\label{fig:Bild2}
\end{figure}

\subsection{Objective and structure of this paper}
Unfortunately, in the case of large data the techniques developed for the case of fixed $n$ do not work anymore. Our main objective is, therefore, to prove limit theorems for the largest interpoint distance $M_{n,p}$ in the high-dimensional case, where $p=p_n\to\infty$ as $n\to\infty$. 

This paper is structured as follows. The main results on the convergence of the maximum are presented in Section \ref{sec:2}. Theorem~\ref{thm:interpoint} asserts that after suitable centering and normalization the largest interpoint distance $M_{n,p}$ converges in distribution  to a standard Gumbel random variable in the high-dimensional regime $p,n\to \infty$. The correlation between the interpoint distances can be expressed in terms of the fourth moment of the entries of the vectors $\x_i$. Interestingly, it turns out that the fluctuations of $M_{n,p}$ might be influenced by this correlation, whereas the first order behavior of $M_{n,p}$ is not (Theorem~\ref{thm:firstord}). 

Theorem~\ref{thm:interpoint} is obtained from the analysis of dependent random walks in Section~\ref{sec:2.2}. In Theorem~\ref{thm:summax}, it is shown that the maximum of these random walks is asymptotically Gumbel distributed under various types of moment assumptions and corresponding growth rates of $p$. We obtain a plethora of ancillary
results such as the joint convergence of maximum and minimum interpoint distances (Theorem~\ref{cor:minimaxi}).

Section \ref{sec:applications} is devoted to geometrical and statistical applications of our findings. First, we generalize the result for the interpoint distances regarding the Euclidean norm to the maximum interpoint distance regarding $q$-norms in Theorem~\ref{thm:qnorm}. Then we propose a test for the equality of means for high-dimensional random vectors based on interpoint distances. In Theorem~\ref{thm:3.3} we show the consistency of this test under the null hypothesis of equal mean vectors and that significant deviations from the null hypothesis will be detected. Section~\ref{sec:3.3} contains an application to maximum-type tests which have gained significant popularity in high-dimensional data analysis. In particular, we study the asymptotic behavior of the largest off-diagonal entry of a sample covariance matrix of iid random vectors from an equicorrelated normal population (Theorem \ref{thm:sample}). 

Finally, we prove in Section \ref{sec:pointproc} that the convergence of the maximum of the dependent random walks can be extended to point process convergence to some Poisson random measure. Among other interesting consequences, this yields the joint distribution of a fixed number of upper order statistics. 
The proof of Theorem \ref{thm:summax} is presented in Section~\ref{sec:mainproof}, while the proofs of the remaining results in Section~\ref{sec:2} are deferred to Section~\ref{sec:otherproofs}. 
In the \hyperref[sec:appendix]{Appendix} 
we collect some useful technical tools.

\subsection{Notation}
Convergence in distribution (resp.\ probability) is denoted by $\cid$ (resp.\ $\cip$) and unless explicitly stated otherwise all limits are for $\nto$. For sequences $(a_n)_n$ and $(b_n)_n$ we write $a_n=O(b_n)$ if $a_n/b_n\leq C$ for some constant $C>0$ and every $n\in\N$, and $a_n=o(b_n)$ if $\lim_{n\to\infty} a_n/b_n=0$. Additionally, we use the notation
$a_n\sim b_n$ if $\lim_{n\to\infty} a_n/b_n=1$ and $a_n\lesssim b_n$ if $a_n$ is smaller than or equal to $b_n$ up to a positive universal multiplicative constant. We further write $a\wedge b:=
  \min \{a,b\}$ for $a,b\in\R$ and for a set $A$ we denote $|A|$ as the number of elements in $A$.

\section{Main results: convergence of the maximum}\label{sec:2}\setcounter{equation}{0}

\subsection{Maximum interpoint distance}\label{sec:maxinterpoint}

We are interested in the limit behavior of the maximum of the interpoint distances,
\begin{equation}\label{eq:defDij}
	D_{ij}^{(2)}:= \| \x_i -\x_j \|^2_2\,, \qquad 1\le i <j\le p\,,
\end{equation} 
where $\x_i=(X_{i1}, \ldots, X_{in})^\top, i=1,\ldots,p$ are $n$-dimensional random vectors, whose components\linebreak $(X_{it})_{i,t\ge 1}$ satisfy the following {\em standard conditions}:
\begin{itemize}
	\item $(X_{it})$ are independent and identically distributed random variables with generic element~$X$.
	\item $\E[X]=0$ and $\E[X^2]=1$. 
\end{itemize}
It is worth mentioning that this is a non-standard extreme value problem since the maximum interpoint distance is a max $U$-statistic. Consequently, the limiting distribution might not necessarily be an extreme value distribution. Further, notice that the mean of the random vectors $\x_i$ has no impact on the distance between the vectors, so we assume it to be zero for simplicity.

In this paper, $p=p_n$ is some integer sequence tending to infinity as $\nto$.
For $\tilde{p}:=p(p-1)/2$ and $y\in\R$ we define
\begin{align}\label{dn1}
	d_{n,1}:=\sqrt{2\log \tilde{p}} - \dfrac{\log\log \tilde{p}+\log 4\pi}{2(2\log \tilde{p})^{1/2}}\qquad\text{and}\qquad
	d_{n}^{(y)}:=d_{n,1}-\frac{y\log \tilde{p}}{3n^{1/2}}\,, \quad n\ge 1.
\end{align}
These sequences will also be used for the appropriate centralization and scaling of $\max_{i<j} D_{ij}^{(2)}$ with the following heuristic explanation. By the central limit theorem (assuming $\E[X^4]<\infty$), the distribution function of 
$\big(D_{12}^{(2)}-\E[D_{12}^{(2)}]\big)/\sqrt{\Var D_{12}^{(2)}}$ converges, as $\nto$, to the standard normal distribution function $\Phi$, where for $i<j$ we have $\E[D_{ij}^{(2)}]= 2n$ and $\Var(D_{ij}^{(2)})=2n(\E[X^4]+1)$.
For an iid \seq\ $(\xi_i)$ of standard normal random variables and $d_{n,1}$ defined as in \eqref{dn1} it holds 
\begin{equation*}
	\lim_{\nto} \P\Big(d_{n,1} \Big(\max_{i=1,\ldots,\tilde{p}} \xi_i-d_{n,1}\Big)\le x\Big)=\exp(-\exp(-x))=:\Lambda(x)\,,
	\qquad x\in\R\,. 
\end{equation*}
The limit \df\ is the standard Gumbel $\Lambda$; see \cite[Example~3.3.29]{embrechts:kluppelberg:mikosch:1997}.
Note that the \seq\ 
$(d_{n,1})$ is chosen \st\  $\tilde{p}\,\ov \Phi(d_{n,1}) \to 1$ as $n\to\infty$,
where $\ov \Phi(x):=1-\Phi(x)$.
Of course $D_{ij}^{(2)}, 1\le i<j\le p$ are not independent random variables. In particular, we have constant correlations
\begin{equation}\label{corr}
\Corr(D_{ij}^{(2)},D_{ik}^{(2)})=\frac{\E[X^4]-1}{2(\E[X^4]+1)}\in[0,1/2)\,, \qquad i<j<k\,,
\end{equation}
and uncorrelatedness if and only if $X$ follows the symmetric Bernoulli distribution $\P(X=-1)=\P(X=1)=1/2$. For large $\tilde{p}$ (relative to the dimension $n$) we have to deal with large number of dependent interpoint distances $D_{ij}^{(2)}$, each of which satisfies a central limit theorem with convergence rate only depending on $n$ and $X$. Therefore, conditions on $X$ and the interplay of $n$ and $p$ are required for the asymptotic behavior of the maximum interpoint distance. Our techniques will rely on Poisson approximation and precise large deviation results in Lemma \ref{lin}, which connects the conditions on the moments of $X$ and the rate of $p=p_n\to \infty$. 
We will assume one of the following four moment conditions:
\begin{enumerate}
	\item[\rm (B1)]
	There exists $s>2$ \st\ $\E[|X|^{2s}(\log(|X|))^{s/2}]< \infty$ and $\E[X^4]\leq 5$ .
	\item[\rm (B2)] 
	There exist constants $\eta>0$ and $0<r\leq 2/3$ \st\ $\E[\exp(\eta\,|X|^{2r})]<\infty$ and $\E[X^4]< 5$.
	\item[\rm (B3)]
	There exist constants $\eta>0$ and $r\geq 1/2$ \st\ $\E[\exp(\eta\,|X|^{2r})]<\infty$ and \mbox{$\E[X^4]= 5$}.
	\item[\rm (B4)]
	There exists a constant $K$ with $\P(|X|\leq K)=1$ and $\E[X^4]= 5$. 
\end{enumerate}

The next theorem is our main result for interpoint distances. 

\begin{theorem}\label{thm:interpoint}
	Let $(\x_i)_{i\le p}$ be iid $\R^n$-valued random vectors, whose components fulfill the standard conditions. Assume one of the conditions {\rm (B1)} -- {\rm (B4)} on $X$ and that 
	$p=p_n\to\infty$ satisfies
	\begin{itemize}
		\item
		$p=O(n^{(s-2)/4})$, if {\rm (B1)} holds.
		\item
		$p=\exp(o(n^{r/(2-r)}))$, if {\rm (B2)} holds.
		\item
		$p=\exp(o(n^{1/(3+2/r)}))$, if {\rm (B3)} holds.
		\item
		$p=\exp(o(n^{1/3}))$, if {\rm (B4)} holds.
	\end{itemize} 
	Then we have 
	$$\max_{1\le i <j\le p} c_n^{(2)} (D_{ij}^{(2)}-b_n^{(2)}) \cid G\,,$$
	where $G$ is standard Gumbel distributed. The sequences $(b_n^{(2)})$ and $(c_n^{(2)})$ are given by
	\begin{align}\label{normconst}
	b_n^{(2)}:=2n+\sqrt{2n(\E[X^4]+1)}\,\tilde{d}_n \quad\text{and}\quad c_n^{(2)}:=\frac{\tilde{d}_n}{\sqrt{2n(\E[X^4]+1)}},
	\end{align}
	where 
	$$\tilde{d}_n:=\begin{cases}d_{n}^{(\tilde{\kappa})}, & \, \text{if (B2) holds with}\,\, r>1/2\,,  \\
		d_{n,1}, & \, \text{otherwise},
	\end{cases}$$ 
	with 
	\begin{align*}
		\tilde{\kappa}:=\frac{\E[X^6]+9\E[X^4]-10(\E[X^3])^2-10}{\sqrt{2}(\E[X^4]+1)^{3/2}}.
	\end{align*}
\end{theorem}
 
	\begin{remark}\em{
 (1) Very recently, \cite{tang2022asymptotic} studied the convergence in distribution of the maximum interpoint distance in the special case $p/n\to c\in(0,\infty)$ and assuming a finite moment generating function of $|X|$. This is a lot more restrictive than the assumptions of Theorem \ref{thm:interpoint}, where in the case $p/n\to c$ only $\E[|X|^{12}(\log|X|)^{3}]< \infty$ is required.\\
 (2) By taking the square root, we see that $\P\Big(\max_{1\le i <j\le p}\Vert \mathbf{x}_i-\mathbf{x}_j\Vert_2<\sqrt{x/c_n^{(2)}+b_n^{(2)}}\Big) \to \Lambda(x)$ for $x\in \R$, as $\nto$.\\
 (3) Instead of considering the largest interpoint distance between all possible combinations of points of one sample, we can study the largest distance between points of two different samples $(\x_i)_{i\le p}$ and $(\y_j)_{j\leq p}$ with the same mean. After similar normalization as in Theorem~\ref{thm:interpoint}, it is shown in Section~\ref{sec:twosample} that $\max_{1\le i,j\le p} \Vert\x_i-\y_j\Vert_2^2$ converges to a Gumbel distributed random variable.\\
 (4) Notice that the assumption $\E[X^4]\leq 5$ in (B1) is equivalent to the correlation in \eqref{corr} being at most $1/3$. The case of correlation larger than $1/3$ will be discussed in Section~\ref{sec:strongcor}. In (B2) we consider exponential moments and require $\E[X^4]<5$. In the special case $\E[X^4]=5$ we need to make stronger assumptions. One possibility is to require a slower rate for $p$ depending on $n$, which we consider in assumption (B3). Alternatively, we can demand stronger assumptions on $X$ such as (B4).}
	\end{remark}

\subsection{Maximum of dependent random walks}\label{sec:2.2}

Theorem \ref{thm:interpoint}  is a direct consequence of Theorem~\ref{thm:summax} below, where more general random walks $\tilde{T}_{ij}$ with the following additive  structure are considered:
\begin{equation}\label{eq:additive}
	\tilde{T}_{ij}:=  \sum_{l=1}^n f(X_{il},X_{jl})\,, \qquad 1\le i<j\le p\,,
\end{equation}
for some measurable function $f:\R^2\to \R$ with $f(x,y)=f(y,x)$. If the random vectors $(\x_i)$ have iid components, then $\tilde{T}_{ij}$ is a sum of iid random variables. This suggests that after appropriate centering and scaling $\tilde{T}_{ij}$ will converge to a standard normal variable. More precisely, for $1\le i<j \le p$ we introduce the standardized sums
\begin{align}\label{eq:defTij}
	T_{ij}:= \frac{1}{\sqrt n} \sum_{l=1}^n Z_{i,j,l},\quad \text{where}\quad Z_{i,j,l}:=\frac{f(X_{il},X_{jl})-\E[f(X_{11},X_{21})]}{\sqrt{\Var(f(X_{11},X_{21}))}}  
\end{align}
are iid (with respect to $l\ge 1$) mean zero, unit variance random variables with generic element $Z$. Define the sequences $(b_n)$ and $(c_n)$ by 
\begin{align}
	b_n:=n\,\E[f(X_{11},X_{21})]+\sqrt{n\Var(f(X_{11},X_{21}))}\,d_{n,1},\quad 
	c_n:=\frac{d_{n,1}}{\sqrt{n\Var(f(X_{11},X_{21}))}},
\end{align}
for $n\geq 1$.
By construction, it holds for $x\in \R$
\begin{align*}
	\P(c_n(\tilde{T}_{ij}-b_n)>x)= \P\Big( T_{ij} > \frac{x}{d_{n,1}}+d_{n,1}  \Big)\,,
\end{align*}
and the central limit theorem yields
$\P\big( T_{ij} > x \big)\sim \ov \Phi(x)\,, \nto$. Note that for an iid sequence $(\xi_i)_{i\ge 1}$ the convergence
\begin{equation*}
	\lim_{\nto} \P\Big(d_{n,1} \Big(\max_{i=1,\ldots,\tilde{p}} \xi_i-d_{n,1}\Big)\le x\Big)=\Lambda(x)\,,
	\qquad x\in\R\,, 
\end{equation*}
is equivalent to (see \cite{resnick:1987})
$$\tilde{p} \, \P(\xi_1>x_n) \sim  \tilde{p} \, \ov \Phi(x_n)\to \e^{-x}\,, \qquad \nto, x\in \R\,,$$
where we used the shorthand notation $x_n=\frac{x}{d_{n,1}}+d_{n,1}$. Hence, it is natural to first establish the corresponding limit relation $\tilde{p}\, \P(T_{12}>x_n)\to \e^{-x}$.

Using Lemma \ref{lin} we are able to find moment conditions on $Z$ under which $\P(T_{12}>x_n) \sim \ov \Phi(x_n)$ as $\nto$. For instance, we get 
\begin{align}\label{eq:edstses}
	\tilde{p}\,\P\big( T_{12} > x_n \big)\sim  \tilde{p}\,\ov \Phi\big(x_n \big)\sim \e^{-x}, \qquad \nto\,
\end{align}
if $p=\exp(o(n^{1/3}))$ and $\E[\exp(\eta\,|Z|^{1/2})]<\infty$ for some $\eta>0$. If we want to choose a larger $p$, we furthermore know by Lemma \ref{lin} that for $p=\exp(o(n^{1/2}))$ and $\E[\exp(\eta\,|Z|^{2/3})]<\infty$ for some $\eta>0$ it holds that
\begin{align}\label{cond2}
	&\tilde{p}\,\P\big( T_{12} > x_n \big) \sim  \tilde{p}\,\ov \Phi\big(x_n \big)\exp\Big(\frac{x_n^3 \, \E[Z^3]}{6n^{1/2}}\Big)\,, \qquad \nto\,.
\end{align}
Therefore, we have to replace the sequence $d_{n,1}$ by $d_{n}^{(\kappa)}$ with $\kappa:=\E[Z^3]$ to get the convergence of \eqref{cond2} to $\e^{-x}$ for $x\in\R$.
\smallskip

Interestingly, the influence of the dependence among the $(T_{ij})_{i<j}$ on the asymptotic distribution of their maximum can be captured in one single correlation parameter
\begin{equation}\label{eq:defrho}
	\rho:=\E[Z_{1,2,1}Z_{1,3,1}] =\Corr(f(X_{11},X_{21}),f(X_{11},X_{31}))\,.
\end{equation}
\begin{remark}\label{rem:rho}\em{
	The range of possible values for $\rho$ is given by $0\le \rho \le 1/2$. This can be shown by checking that the covariance matrix of the random variables $(T_{ij})_{1\le i <j}$ is positive semidefinite if and only if $0\le \rho \le 1/2$. 
	
	To this end, note that for $d\in \N$ the covariance matrix of $(T_{12},T_{13},\ldots, T_{1,d+1})$ is given by $\Sigma_d=\rho {\bf1}_d{\bf1}_d^{\top} +(1-\rho) \bfI_d$, where ${\bf1}_d$ is the $d$-dimensional vector of ones. Since $\rho$ is a correlation coefficient we must have $\rho\in [-1,1]$. It is well--known that $\Sigma_d$ is positive semidefinite if and only if $\rho \ge -1/(d-1)$. Since $d$ was arbitrary, we deduce that $\rho\ge 0$. 
	Next, one can check that the covariance matrix $M_3$ of $(T_{12},T_{13},T_{23},T_{14},T_{24},T_{34})$ has an eigenvalue $1-2\rho$ and therefore $\rho \le 1/2$ in order for $M_3$ to be positive semidefinite.
}\end{remark}

We will assume one of the following four moment conditions, where we recall that $Z\eid Z_{1,2,1}$ with $Z_{1,2,1}$ as in \eqref{eq:defTij}:
\begin{enumerate}
	\item[\rm (C1)]
	There exists $s>2$ \st\ $\E[|Z|^s(\log(|Z|))^{s/2}]< \infty$ and $\rho\leq 1/3$.
	\item[\rm (C2)] 
	There exist constants $\eta>0$ and $0<r\leq 2/3$ \st\ $\E[\exp(\eta\,|Z|^r)]<\infty$ and $\rho<1/3$.
	\item[\rm (C3)]
	There exist constants $\eta>0$ and $r\geq 1/2$ \st\ $\E[\exp(\eta\,|Z|^r)]<\infty$ and $\rho=1/3$.
	\item[\rm (C4)] 
	There exists a constant $K$ with $\P(|Z|\leq K)=1$ and $\rho=1/3$. 
\end{enumerate}
Our next result, Theorem \ref{thm:summax}, provides conditions for the convergence of the maximum
$$\max_{1\le i<j\le p} d_n(T_{ij}-d_n)=\max_{1\le i<j\le p} c_n(\tilde{T}_{ij}-b_n)\,,$$
where
$$
d_n:=\begin{cases}d_{n}^{(\kappa)}, & \, \text{if {\rm (C2)} holds with}\,\, r>1/2,  \\
	d_{n,1}, & \, \text{otherwise.}
\end{cases}
$$

\begin{theorem}\label{thm:summax}
	Let $(X_{it})_{i,t\ge 1}$ be iid random variables and let $Z\eid Z_{1,2,1}$ with $Z_{1,2,1}$ as in \eqref{eq:defTij}. Furthermore, assume one of the conditions {\rm (C1)} -- {\rm (C4)} on $Z$ and that 
	$p=p_n\to\infty$ satisfies
	\begin{itemize}
		\item
		$p=O(n^{(s-2)/4})$, if {\rm (C1)} holds.
		\item
		$p=\exp(o(n^{r/(2-r)}))$, if {\rm (C2)} holds.
		\item
		$p=\exp(o(n^{1/(3+2/r)}))$, if {\rm (C3)} holds.
		\item
		$p=\exp(o(n^{1/3}))$, if {\rm (C4)} holds.
	\end{itemize} 
	Then
	\begin{equation}\label{eq:maxf}
		\max_{1\le i<j\le p} d_n(T_{ij}-d_n)\cid G, \qquad \nto\,,
	\end{equation}
	where $G$ is standard Gumbel distributed.
\end{theorem}

\begin{proof}[Sketch of the proof]
	We restrict ourselves to the proof under condition {\rm (C4)}. While this case might appear as the easiest of the four, it deals with the largest correlation $\rho=1/3$ and explains why this particular value plays a special role. 
	
	Firstly, \eqref{eq:edstses} already establishes the necessary and sufficient condition for the convergence of the maximum of $\tilde{p}$ iid copies of $T_{12}$. 
	More precisely, letting $x_n:=x/d_n+d_n$ for $x\in \R$, an application of Lemma \ref{lin}(iii) yields for $p=\exp(o(n^{1/3}))$ that
	\begin{equation}\label{eq:drhdfdf}
		\begin{split}
			\tilde{p}\, \P\Big( \frac{1}{\sqrt n} \sum_{l=1}^n Z_{1,2,l} > x_n \Big)&=\tilde{p}\, \ov \Phi(x_n)\exp\left(\frac{x_n^3\E[Z^3]}{6n^{1/2}}\right)\left(1+O\left(\frac{1+x_n}{n^{1/2}}\right)\right)\\
			&\sim \tilde{p}\, \ov \Phi(x_n)\to \e^{-x} \,, \qquad \nto\,.
		\end{split}
	\end{equation}
	Notice that this convergence does not hold if $p=\exp(n^{1/3})$ and $\E[Z^3]\neq 0$. This means that the growth rate $p=\exp(o(n^{1/3}))$ cannot be increased without adjusting the normalization; c.f.\ Remark~\ref{rem:2.6}. We would like to point out that similar optimality properties hold in all cases in Theorem~\ref{thm:summax}.
	
	Secondly, combining \eqref{eq:drhdfdf} and Lemma \ref{lem:poissonapprox2} we may conclude the desired result \eqref{eq:maxf} if we can show that
	\begin{equation}\label{eq:A20102}
		\lim_{\nto} p^3\,\P\big( d_n(T_{12}-d_n)>x, d_n(T_{13}-d_n)>x\big)=0\,.
	\end{equation}
	As the $|Z_{i,j,l}|$ are bounded by $K$, we can apply Theorem 1.1 of \cite{Zaitsev:1987} to obtain an approximation of the distribution of the vector $(T_{12}, T_{13})$ by the distribution of a vector $(N_1,N_2)$ of standard normal variables with $\Cov(N_1,N_2)=\rho$. From this approximation we deduce that for $\lambda_n=(\log p)^{-1/2}$ 
	\begin{align*}
		&p^3\,\P \big(\min(T_{12}, T_{13})>x_n\big)\leq p^3\,\P(\min(N_1,N_2)>x_n-\lambda_n)+ p^3\,c_1\exp\Big(-c_2\frac{\sqrt{n}\lambda_n}{K}\Big),
	\end{align*}
	where $c_1,c_2$ are absolute constants. 
 Using properties of the tails of equicorrelated Gaussian random variables in Lemma~\ref{lem:1}, the terms on the \rhs\ converge to zero provided that $\rho\le 1/3$ and $p= \exp(o( n^{1/3}))$. 
	
	It is worth mentioning that $p^3\,\P(\min(N_1,N_2)>x_n)\to 0$ if and only if $\rho\le 1/3$. The arguments under conditions {\rm (C1)}-{\rm (C3)} are similar but more involved. Unfortunately, since $Z$ is not bounded in these cases a series of truncation steps is required. The detailed proof of Theorem \ref{thm:summax} is given in Section \ref{sec:mainproof}.
\end{proof}
An important step in the sketch of the proof is the normal approximation of the sum of independent random variables. The independence requirement on the components of the random vectors $\x_i$ can be weakened
	if we accept stronger conditions on the moments of the components or the rate of $p$. For example, using the moderate deviation result for locally dependent random variables in \cite[Theorem~2.1]{liu2021cramer} and following the lines of the proof of Theorem \ref{thm:summax} under (C2) we can show \eqref{eq:maxf} for locally dependent components if we demand $p=\exp(o(n^{1/3}))$ and a moment condition, which is stronger than $\E[\exp(\eta\,|Z|)]<\infty$ and determined by the dependence of the components.
\begin{remark}\label{rem:29}\em{
	In the proof of Theorem \ref{thm:summax} we employ the Chen-Stein Poisson
	approximation method from \cite{arratia:goldstein:gordon:1989}. For $x\in \R$ consider the sums
	\begin{align*}
		W_n(x):=\sum_{1\le i<j\le p}  \1_{\{d_n(T_{ij}-d_n) >x\}}\,, \qquad n \ge 1.
	\end{align*}
	Along the lines of the proof Theorem \ref{thm:summax} it can be shown that $W_n(x)\cid W(x)$, where $W(x)$ is a Poisson distributed random variable with parameter $\e^{-x}$. It is easy to see 
	$$\{W_n(x)=0\} = \left\{\max_{1\le i<j\le p} d_n(T_{ij}-d_n) \le x \right\}\,, \qquad x\in \R\,,$$
	so that the convergence in distribution of $W_n(x)$ yields 
	$$\lim_{\nto} \P(W_n(x)=0) =\exp\big(-\e^{-x}\big)\,,\qquad x\in \R\,,$$
	which in turn implies \eqref{eq:maxf}. Finally, we mention that in Section~\ref{sec:pointproc} the convergence of the maximum will be extended to point process convergence (Theorem~\ref{thm:summaxpp}).
}\end{remark}

We proceed by discussing the assumptions of Theorem \ref{thm:summax}. One can see that the rate of $p\to\infty$ is connected to conditions on the moments of $|Z|$. The larger $p$ is relatively to $n$ the more moments have to exist to obtain \eqref{eq:maxf}. Intuitively this makes a lot of sense as a large $p$ increases the number of $T_{ij}$'s in the maximum, but does not improve the rate of convergence of the $T_{ij}$'s to the normal distribution. If $c_1 n\leq p\leq c_2 n$ holds for constants $c_1$ and $c_2$, then \cite{li:liu:rosalsky:2010} and \cite{li:qi:rosalsky:2012} proved that $\E[Z^{6-\delta}]<\infty$ for every $\delta>0$ is a necessary condition in the case $f(x,y)=xy$ (see \eqref{eq:additive}). According to our Theorem \ref{thm:summax} a sufficient condition in this case is $\E[Z^6\log(|Z|^3)]<\infty$ showing that the moment condition {\rm(C1)} cannot be weakened in general.

If $p$ grows exponentially in $n$, then we need finite exponential moments of certain powers of $|Z|$.
If, for instance, $\E[\exp(\eta\,|Z|^{1/2})]<\infty$, we get \eqref{eq:maxf} for $p=\exp(o(n^{1/3}))$ provided that $\rho<1/3$. If $\rho=1/3$, one has to either reduce the range to $p=\exp(o(n^{1/7}))$ or assume that $|Z|$ is bounded.

Noting that, for $0<s<t$, $\E[\exp(\eta\,|Z|^t)]<\infty$ implies $\E[\exp(\eta\,|Z|^s)]<\infty$, Theorem~\ref{thm:summax} generalizes several special cases known in the literature, where typically a finite moment generating function of $|Z|$ is assumed and the maximum is taken over uncorrelated terms, that is, $\rho=0$; see \cite{jiang:2004b, li:liu:rosalsky:2010, han:chen:liu:2017, heiny:mikosch:yslas:2021}.

\begin{remark}\label{rem:2.6}\em{
	Under {\rm (C2)} choosing $r>1/2$ equation \eqref{eq:maxf} does not hold for $d_n=d_{n,1}$ anymore. For example for $\log p=n^{1/3}$ we get
	\begin{align*}
		\lim\limits_{n\to\infty}\P\Big(\max_{1\leq i<j\leq p}d_{n,1}\big(T_{ij}-d_{n,1}\big)\leq x\Big)=\exp(-\e^{-x+\frac{4}{3}\E[Z^3]})\,, \qquad x\in \R.
	\end{align*}
	We have to recenter $\max_{1\leq i<j\leq p}T_{ij}$ by $d_{n}^{(\E[Z^3])}$ to obtain \eqref{eq:maxf}, compare also Shao and Zhou \cite[p.~625-626]{shao2014necessary}. 
}\end{remark}
Our next result concerns the order of the maximum.
\begin{proposition}\label{cor:opt}
	\begin{enumerate}
		\renewcommand{\labelenumi}{(\roman{enumi})}
		\item
		Under each of the conditions of Theorem \ref{thm:summax} we have
		\begin{align}\label{conv}
			\max_{1\leq i<j\leq p}\frac{1}{\sqrt{\log p}}\, T_{ij}\to 2
		\end{align}
		in probability as $n\to\infty$.
		\item
		Let $0<s\leq 1/2$. If \eqref{conv} holds for  any $p$ satisfying $\log p = o(n^s)$ as $n\to\infty$, then\linebreak $\E[\exp(\eta|Z|^\frac{2s}{1+s})]<\infty$ for some $\eta>0$.
	\end{enumerate}
\end{proposition}
From Proposition \ref{cor:opt}(ii) we conclude that an exponential growth of $p$ requires a finite moment generating function of some power of $|Z|$ which is in line with conditions {\rm (C2)-(C4)} of Theorem~\ref{thm:summax}. In particular, it is impossible to replace {\rm (C2)-(C4)} by the weaker condition {\rm (C1)}.

By Theorem \ref{thm:summax} we are able to draw conclusions for the minimum of the random walks $T_{ij}$, for example the minimum interpoint distance.
\begin{corollary}
	Under the conditions of Theorem \ref{thm:summax} we get
	\begin{equation}
		\min_{1\le i<j\le p} d_n(T_{ij}+d_n) \cid -G,\qquad \nto\,,
	\end{equation}
	where $G$ is standard Gumbel distributed.
\end{corollary}
\begin{proof}
	The result is an immediate consequence of Theorem \ref{thm:summax} as 
	\begin{align*}
		\min_{1\le i<j\le p}T_{ij}+d_n=-\max_{1\le i<j\le p}-(T_{ij}+d_n).
	\end{align*}
\end{proof}

Additionally, it turns out that the normalized maxima and minima are asymptotically independent.

\begin{theorem}\label{cor:minimaxi}
	Under the conditions of Theorem \ref{thm:summax} 
	we have 
	\begin{align*}
		&\Big(\max_{1\le i<j\le p} d_n(T_{ij}-d_n), \min_{1\le i<j\le p} d_n(T_{ij}+d_n)\Big)\cid (G,-G'),\qquad \nto\,,
	\end{align*}
	where $G$ and $G'$ are independent standard Gumbel distributed random variables.
\end{theorem}

\subsection{Maximum of dependent random walks in the case of strong correlation} \label{sec:strongcor}

In Remark~\ref{rem:rho}, we showed that $0\le \rho \le 1/2$ whereas the results so far were restricted to $\rho \in [0,1/3]$. The sketch of the proof of Theorem~\ref{thm:summax} gives a first explanation as to why the case $\rho >1/3$ is different. Next, we provide some error bounds for the convergence in \eqref{eq:maxf} under condition {\rm(C4)}.
\begin{proposition}\label{thm:rate}
	If $p=\exp(o(n^{1/3}))$ and there exists a constant $K$ with $|Z|<K$, then it holds, as $\nto$, 
	\begin{align*}
		&\big|\P\big(\max_{1\leq i<j\leq p}d_{n,1}(T_{ij}-d_{n,1})\leq x\big)-\exp(-\e^{-x})\big|\\
		&= O((\log p)^{-\rho/(1+\rho)}p^{-(1-3\rho)/(1+\rho)})+ O\Big(\sqrt{\frac{(\log p)^3}{n}}\Big)+O\Big(\frac{(\log(\log p))^2}{\log p}\Big)\,, \qquad x\in \R.
	\end{align*}
\end{proposition}
We see that the first error term does not vanish asymptotically if $\rho>1/3$. Therefore we need an alternative approach.

To prove Theorem \ref{thm:summax} under optimal assumptions we applied Theorem 1.1 of \cite{Zaitsev:1987}, which provides a pointwise normal approximation of the random variables $T_{ij}$. Another possibility is to use the normal approximation for maxima of sums of high-dimensional random vectors in \cite{chernozhukov2013gaussian}, which yields a bound on the Kolmogorov distance between the maximum of the $T_{ij}$ and the maximum of normal distributed random variables. For this method stronger moment assumptions and more restrictions on the rate of $p$ are necessary but the advantage is that this bound does not limit the dependency structure. Therefore, we can make statements about the maximum of the $T_{ij}$ even if $\rho>1/3$. 

To this end, consider a random field $(Y_{ij})_{1\le i <j}, i,j\in \N$ of standard normal random variables $Y_{ij}$. For $i<j$ and $s<t$ assume that
\begin{align}\label{eq:Yijsdf}
	\Cov(Y_{ij},Y_{st})=\begin{cases} 0,\quad\quad\quad\quad\lbrace i,j\rbrace\cap \lbrace s,t\rbrace=\emptyset\\
		1,\quad\quad\quad\quad \lbrace i,j\rbrace=\lbrace s,t\rbrace\\
		\rho,\quad\quad\quad\quad \text{otherwise}.
	\end{cases}
\end{align}
The range of possible values for $\rho$ is given by $0\le \rho \le 1/2$. This can be shown by checking that the covariance matrix of the random variables $(Y_{ij})_{1\le i <j}$ is positive semidefinite if and only if $0\le \rho \le 1/2$; see Remark~\ref{rem:rho}. Indeed, the covariance matrices of $(Y_{ij})_{1\le i<j\le p}$ and $(T_{ij})_{1\le i<j\le p}$ are the same.
By Corollary 2.1 of \cite{chernozhukov2013gaussian} we get the following lemma.
\begin{lemma}\label{lem:norm}
	Assume there exist constants $\eta, c,C>0$ such that 
	\begin{align}\label{eq:lem2.11}
		\E[\exp(\eta |Z|)]<\infty \qquad \text{and}\qquad\frac{(\log(\tilde{p}n))^7}{n}\leq C{}n^{-c}\,,
	\end{align}
	then it holds that
	\begin{align*}
		\sup_{t\in\R}\Big|\P\Big(\max_{1\leq i<j\leq p}\frac{1}{\sqrt{n}}\sum_{l=1}^n Z_{i,j,l}\leq t\Big)-\P\Big(\max_{1\leq i<j\leq p} Y_{ij}\leq t\Big)\Big|\leq C{}n^{-c}.
	\end{align*}
\end{lemma}
If \eqref{eq:lem2.11} is fulfilled, we obtain that the limiting distributions of $\max T_{ij}$ and $\max Y_{ij}$ are the same. More precisely, it holds
\begin{align}\label{eq:sdsdsxd}
	\sup_{t\in\R}\Big|\P\Big(\max_{1\leq i<j\leq p} d_{n,1}(T_{ij}-d_{n,1})  \leq t\Big)-\P\Big(\max_{1\leq i<j\leq p} d_{n,1}(Y_{ij}-d_{n,1}) \leq t\Big)\Big|\to 0\,, \qquad \nto\,.
\end{align}

\begin{remark}\em{
	The second condition in \eqref{eq:lem2.11} holds if $p=\exp(o(n^{1/7-\vep}))$ for $\vep>0$. For comparison, in the case $\rho<1/3$  Theorem~\ref{thm:summax} provides a result for the larger rate $p=\exp{o(n^{1/3})}$ under the weaker moment assumption $\E[\exp(\eta |Z|^{1/2} )]<\infty$ for some $\eta>0$. 
}\end{remark}
In view of \eqref{eq:sdsdsxd}, it is natural to ask for which values $\rho$ a Gumbel limit can be achieved. 
By similar arguments as in the proof of Proposition \ref{thm:rate} we get for $x\in\R$ that
\begin{align*}
	&\Big|\P\Big(\max_{1\leq i<j\leq p} d_{n,1}(Y_{ij}-d_{n,1})\leq x\Big)-\exp(-\e^{-x})\Big|\\
	&= O((\log p)^{-\rho/(1+\rho)}p^{-(1-3\rho)/(1+\rho)})+O\Big(\frac{(\log \log p)^2}{\log p}\Big),
\end{align*}
as $\nto$. Unfortunately, this does not yield a positive result for $\rho>1/3$. To proceed, we define the random variables
$$V_{n,\rho}(x):=\sum_{1\le i<j\le p}  \1_{\{d_{n,1}(Y_{ij}-d_{n,1}) >x\}}\,, \qquad n \ge 1, x\in \R,$$
where the $(Y_{ij})$ are as in \eqref{eq:Yijsdf}.
For the same reasons as in Remark \ref{rem:29} it holds that 
$$V_{n,\rho}(x)\cid W(x)\,, \qquad \nto\,, \rho\in [0, 1/3]\,,$$ 
where $W(x)$ is a Poisson distributed random variable with parameter $\e^{-x}$. Recall that the Poisson distribution is uniquely characterized by its sequence of moments; see e.g.\ \cite{kuba:panholzer:2016}. Our next result reveals the asymptotic behaviors of the first two moments of $V_{n,\rho}(x)$.
\begin{proposition}\label{prop:momentsv}
	For $\rho\in [0,1/2]$ and $x\in \R$ we have, as $\nto$,
	\begin{align*}
		\E[V_{n,\rho}(x)]&\to \E[W(x)] =\e^{-x}\,,\\
		\E[V_{n,\rho}^2(x)]&\sim \begin{cases} 
			\E[W(x)^2]=\e^{-x}+\e^{-2x} ,\quad\quad\quad &\rho\le 1/3,\\
			\frac{(1+\rho)^{3/2}}{8\pi (1-\rho)^{1/2}}\, (\log p)^{-\rho/(1+\rho)}p^{(3\rho-1)/{(1+\rho)}},\quad\quad\quad &\rho>1/3.
		\end{cases}
	\end{align*}
\end{proposition}
In this sense we observe a phase transition at $\rho=1/3$.

Our final goal is to show that the first order behavior of $\max T_{ij}$ is the same for all $\rho\in[0,1/2]$. The following result generalizes Proposition~\ref{cor:opt} to all possible values of $\rho$.
\begin{theorem}\label{thm:firstord}
	Let $(X_{it})_{i,t\ge 1}$ be iid random variables and let $Z\eid Z_{1,2,1}$ with $Z_{1,2,1}$ as in \eqref{eq:defTij}. Assuming \eqref{eq:lem2.11} it holds for all $\rho\in[0,1/2]$ that 
	\begin{align}\label{conv13}
		\max_{1\leq i<j\leq p}\frac{1}{\sqrt{\log p}}\, T_{ij}\cip 2\,, \qquad \nto\,.
	\end{align}
\end{theorem}
In other words, there is no phase transition in the first order behavior of the maximum of the $T_{ij}$. The main step in the proof of Theorem~\ref{thm:firstord} consists in establishing \eqref{conv13} with $T_{ij}$ replaced by $Y_{ij}$.

\section{More applications}\label{sec:applications} \setcounter{equation}{0}

We recall that $\x_i=(X_{i1}, \ldots, X_{in})^\top, i=1,\ldots,p$, are $n$-dimensional random vectors, whose components $(X_{it})_{i,t\ge 1}$ satisfy the {\em standard conditions}. Throughout this section, we are going to work with the sequence $d_n:=d_{n,1}$ defined in \eqref{dn1}.
\subsection{ \texorpdfstring{$q$}{}-norms}
Instead of using the Euclidean norm (or $2$-norm) to investigate the maximum of the interpoint distances, one can investigate more general $q$-norms. For $q\geq 1$ the $q$-norm of a vector $\x=(X_1,\ldots,X_n)^{\top}\in\R^n$ is defined by
\begin{align*}
	\Vert\x\Vert_q:=\Big(\sum_{i=1}^n |X_i|^q\Big)^{1/q}.
\end{align*} 
For $0<\E[|X|^{3q}]<\infty$ and $p=o(n^{1/5}(\log n)^{-6/5})$, as $n\to\infty$, Biau and Mason \cite[Proposition~5]{biau2015high} obtained for appropriate sequences $g_n$ and $h_n$ that
\begin{align}\label{eq:biaumason}
	g_n\Big(\max_{1\leq i\leq p}\Vert \x_i\Vert_q-\min_{1\leq j\leq p}\Vert \x_j\Vert_q\Big)-h_n\cid G+G',
\end{align}
where $G$ and $G'$ are independent, standard Gumbel distributed random variables. Accordingly, they did not investigate the asymptotic behavior of the maximum interpoint distance but the asymptotic behavior of the difference between the distance of the origin to its farthest and nearest neighbors which is also known as {\it contrast} in the computational learning literature and is an important statistic for high-dimensional data processing.

In order to consider the interpoint distances regarding $q$-norms, we set for $q\geq 1$
\begin{equation}\label{eq:defQij}
	D_{ij}^{(q)}:= \| \x_i -\x_j \|_q^q\,, \qquad 1\le i <j\le p\,,
\end{equation}  
and let
\begin{align*}
	\hat\rho:=\Corr\big(|X_{11}-X_{21}|^q, |X_{11}-X_{31}|^q\big)=\Corr\big(D_{12}^{(q)},D_{13}^{(q)}\big)\,.
\end{align*}
Similarly to Theorem~\ref{thm:interpoint}, we will need the following four moment conditions:
\begin{enumerate}
	\item[\rm (D1)]
	There exists $s>2$ \st\ $\E[|X|^{qs}(\log(|X|))^{s/2}]< \infty$ and $\hat\rho\leq 1/3$.
	\item[\rm (D2)] 
	There exist constants $\eta>0$ and $0<r\leq 1/2$ \st\ $\E[\exp(\eta\,|X|^{qr})]<\infty$ and $\hat\rho<1/3$.
	\item[\rm (D3)]
	There exist constants $\eta>0$ and $r\geq 1/2$ \st\ $\E[\exp(\eta\,|X|^{qr})]<\infty$ and $\hat\rho=1/3$.
	\item[\rm (D4)]
	There exists a constant $K$ with $|X|<K$ and $\hat\rho=1/3$. 
\end{enumerate}
The next result is a generalization of Theorem~\ref{thm:interpoint} to $q$-norms.
\begin{theorem}\label{thm:qnorm}
	Let $q\ge 1$ and assume one of the conditions {\rm (D1)}\,--\,{\rm (D4)} on $X$ and that 
	$p=p_n\to\infty$ satisfies
	\begin{itemize}
		\item
		$p=O(n^{(s-2)/4})$ if {\rm (D1)} holds.
		\item
		$p=\exp(o(n^{r/(2-r)}))$ if {\rm (D2)} holds.
		\item
		$p=\exp(o(n^{1/(3+2/r)}))$ if {\rm (D3)} holds.
		\item
		$p=\exp(o(n^{1/3}))$ if {\rm (D4)} holds.
	\end{itemize} 
	Then it holds that 
	$$\max_{1\le i <j\le n} c_n^{(q)} \big(D_{ij}^{(q)}-b_n^{(q)}\big) \cid G\,,$$
	where $G$ is standard Gumbel distributed and the sequences $(b_n^{(q)})$ and $(c_n^{(q)})$ are given by
	\begin{align*}
		b_n^{(q)}&:=n\,\E[|X_{11}-X_{21}|^q]+\sqrt{n\Var(|X_{11}-X_{21}|^{q})}d_n\,,\\
		c_n^{(q)}&:=\frac{d_n}{\sqrt{n\Var(|X_{11}-X_{21}|^{q})}}\,.
	\end{align*}
\end{theorem}

\begin{proof}
	Theorem \ref{thm:qnorm} is a special case of Theorem \ref{thm:summax} as it holds that
	\beam
	\frac{D_{ij}^{(q)}-n\,\E[|X_{11}-X_{21}|^q]}{\sqrt{n\Var(|X_{11}-X_{21}|^{q})}} = \frac{1}{\sqrt{n}}\sum_{l=1}^n Z_{i,j,l},
	\eeam
	where
	\begin{align*}
		Z_{i,j,l}:=\frac{|X_{il}-X_{jl}|^q-\E[|X_{11}-X_{21}|^q]}{\sqrt{\Var(|X_{11}-X_{21}|^{q})}}\qquad 1\le i <j\le p\,;\,\, 1\leq l\leq n.
	\end{align*} 
	We have $\Cov(Z_{i,j,1},Z_{r,s,1})=\hat\rho\leq \frac{1}{3}$ for $|\{i,j,r,s\}|=3$.
	Additionally, the conditions {\rm (C1)}--{\rm (C4)} for Theorem \ref{thm:summax} follow from {\rm (D1)}--{\rm (D4)}.
\end{proof}

Using Theorem \ref{cor:minimaxi} one can deduce the following result for the maximum and minimum $q$-norm distances.

\begin{corollary}\label{cor:sdf}
	Under the assumptions of Theorem \ref{thm:qnorm} it holds, as $\nto$,
	\begin{align*}
		&\big(\max_{1\le i<j\le p} c_n^{(q)} \big(D_{ij}^{(q)}-b_n^{(q)}\big), \min_{1\le i<j\le p} c_n^{(q)} \big(D_{ij}^{(q)}-b_n^{(q)}\big)+2d_n^2\big)\cid (G,-G')\,,
	\end{align*}
	where $G$ and $G'$ are independent, standard Gumbel distributed random variables.
\end{corollary}
In combination with the continuous mapping theorem, we obtain from Corollary \ref{cor:sdf}
\begin{align*}
	c_n^{(q)} \Big(\max_{1\le i <j\le n} D_{ij}^{(q)}-\min_{1\le i <j\le n} D_{ij}^{(q)}\Big)-2d_n^2\cid G+G',\qquad \nto\,,
\end{align*}
which is of similar structure as \eqref{eq:biaumason}.

\subsection{Testing the equality of means for high-dimensional vectors}

We consider high-dimensional observations $\tilde{\x}_1, \ldots, \tilde{\x}_p$ of the form 
$$\tilde{\x}_i=\x_i+\mu_i\,,$$
where $(\x_i)_{i=1,\ldots, p}$ are iid random vectors whose components fulfill the standard conditions and $(\mu_i)_{i=1,\ldots, p}$ are some vectors in $\R^n$. We assume that $\E[X^4]$, the fourth moment of the components, is finite and known. Since $\x_1$ is centered, the mean vector of $\tilde \x_i$ is given by $\E[\tilde \x_i]=\mu_i$. 
We are interested in testing the equality of the mean vectors $\mu_i$.
The corresponding testing problem is formulated by the null and alternative hypotheses
\begin{equation}\label{meantest}
	\begin{split}
		&H_0:\, \mu_1=\mu_2=\ldots=\mu_p, \\
		&H_A:\,\mu_{i^*}\neq\mu_{j^*} \quad \text{for some } 1\le i^*<j^*\le p\,.
	\end{split}
\end{equation}
Our test statistic $\mathcal{T}_n$ is going to be the maximum interpoint distance of the $(\tilde{\x}_i)$, that is,
\begin{equation}\label{eq:defDijtil}
	\mathcal{T}_n:=\max_{1\le i <j\le p} \tilde{D}_{ij}^{(2)}:=\max_{1\le i <j\le p} \| \tilde{\x}_i -\tilde{\x}_j \|_2^2\,.
\end{equation}
Observing that, under $H_0$, we have $\tilde{D}_{ij}^{(2)}={D}_{ij}^{(2)}$, the asymptotic distribution of the test statistic $\mathcal{T}_n$ is stated in Theorem \ref{thm:interpoint}, namely 
$c_n^{(2)} (\mathcal{T}_n-b_n^{(2)})$ converges to a standard Gumbel distributed random variable, where 
\begin{equation}\label{eq:asse}
	b_n^{(2)}:=2n+\sqrt{2n(\E[X^4]+1)}\,\tilde{d}_n\,, \quad\quad c_n^{(2)}:=\frac{\tilde{d}_n}{\sqrt{2n(\E[X^4]+1)}},
\end{equation}
and $\tilde{d}_n$ as in Theorem \ref{thm:interpoint}.
We remark that $\E[X^4]$ in \eqref{eq:asse} can be replaced by an estimate for $\E[X^4]$ such as the empirical fourth moment. For brevity of presentation this will not be pursued further.

The null hypothesis
in \eqref{meantest} is rejected, whenever 
\begin{align}	 \label{hd5}	
	\mathcal{T}_{n}>\frac{q_{1-\alpha}}{c_n^{(2)}}+b_n^{(2)}\,,
\end{align} 
where $q_{1-\alpha} = -\log(\log(\frac{1}{1-\alpha})) $ is the 
$({1-\alpha})$-quantile of the standard Gumbel distribution with distribution function $\exp ( - \exp ( - x) ), x\in \R$. The next result shows that this test has asymptotic level $\alpha\in (0,1)$ and analyzes its behavior under the alternative.

\begin{theorem}\label{thm:3.3}
	Assume the conditions of Theorem \ref{thm:interpoint}. Under the null hypothesis $H_0$, it holds for any $\alpha\in (0,1)$
	\begin{align}\label{eq:underH0}
		\lim_{\nto} \P\Big(\mathcal{T}_{n}>\frac{q_{1-\alpha}}{c_n^{(2)}}+b_n^{(2)}\Big) = \alpha\,.
	\end{align}
	Under the alternative hypothesis $H_A$, assume that there exist integer sequences $(i^*_n)_{n\ge 1}$, $(j^*_n)_{n\ge 1}$ satisfying $1\le i^\ast_n<j^\ast_n\le p$ such that
	\begin{align}\label{con}
		\lim_{\nto} \frac{\Vert \mu_{i^*_n}-\mu_{j^*_n}\Vert^2_2}{\sqrt{n \, \max(\log p,n)}} = \infty\,.
	\end{align}
	Then it holds for any $\alpha\in (0,1)$
	\begin{align*}
		\lim_{\nto} \P\Big(\mathcal{T}_{n}>\frac{q_{1-\alpha}}{c_n^{(2)}}+b_n^{(2)}\Big) = 1\,.
	\end{align*}
\end{theorem}
Theorem~\ref{thm:3.3} states that the test \eqref{hd5} is consistent under the null hypothesis $H_0$. Moreover, significant deviations (in the sense of \eqref{con}) from $H_0$ will always be detected by this test. We remark that conditions such as \eqref{con} are quite common for maximum-type tests (see \cite{cai:liu:2011,  goesmann:stoehr:heiny:dette:2022}).
\begin{proof}[Proof of Theorem~\ref{thm:3.3}]
	Assertion \eqref{eq:underH0} follows from the fact that $c_n^{(2)} (\mathcal{T}_n-b_n^{(2)})$ converges in distribution to a standard Gumbel random variable.
	
	Let us turn to $H_A$. For simplicity we will write $i^*,j^*$ instead of $i^*_n, j^*_n$, respectively.
	Using the definition of $\mathcal{T}_n$ we have
	\begin{align*}
		\mathcal{T}_n&=\max_{1\le i<j\le p}\Vert \mu_{i}-\mu_{j}+\x_{i}-\x_{j} \Vert_2^2\\
		&\geq \Vert \mu_{i^*}-\mu_{j^*}+\x_{i^*}-\x_{j^*} \Vert_2^2\\
		&\geq \Vert \mu_{i^*}-\mu_{j^*}\Vert_2^2-2\Vert \mu_{i^*}-\mu_{j^*}\Vert_2\Vert\x_{i^*}-\x_{j^*} \Vert_2 +\Vert\x_{i^*}-\x_{j^*} \Vert_2^2\,.
	\end{align*}
	Setting $g_{n,\alpha}:=q_{1-\alpha}/c_n^{(2)}+b_n^{(2)}-2n$, we thus get for $\alpha\in (0,1)$ and $\vep>0$
	\begin{align*}
		&\P\Big(\mathcal{T}_{n}>\frac{q_{1-\alpha}}{c_n^{(2)}}+b_n^{(2)}\Big)
		=\P\Big(\mathcal{T}_{n}-2n>g_{n,\alpha}\Big)\\
		&\geq \P\Big(\Vert \mu_{i^*}-\mu_{j^*}\Vert_2^2-2\Vert \mu_{i^*}-\mu_{j^*}\Vert_2\Vert\x_{i^*}-\x_{j^*} \Vert_2 +\Vert\x_{i^*}-\x_{j^*} \Vert_2^2-2n>g_{n,\alpha}\Big)\\
		&\geq \P\Big(\Vert \mu_{i^*}-\mu_{j^*}\Vert_2^2-2\Vert \mu_{i^*}-\mu_{j^*}\Vert_2\Vert\x_{i^*}-\x_{j^*} \Vert_2 >(1+\vep) g_{n,\alpha}\Big)\\
		&\quad -\P\Big(\Vert\x_{i^*}-\x_{j^*} \Vert_2^2-2n \le -\vep\, g_{n,\alpha}\Big)\\
		&=: P_1-P_2 \,.
	\end{align*}
	Using the definitions of $b_n^{(2)}, c_n^{(2)}$ in \eqref{eq:asse}, we get
	\begin{align*}
		g_{n,\alpha}=\frac{q_{1-\alpha}}{c_n^{(2)}}+b_n^{(2)}-2n=
		\frac{q_{1-\alpha} \sqrt{2n(\E[X^4]+1)}}{\tilde{d}_n}+\sqrt{2n(\E[X^4]+1)}\,\tilde{d}_n\,.
	\end{align*}
	Since by the central limit theorem
	\begin{equation}\label{eq:clllt}
		\frac{\Vert\x_{i^*}-\x_{j^*} \Vert_2^2-2n}{\sqrt{2n(\E[X^4]+1)}} \cid \mathcal{N}(0,1)\,, \qquad \nto\,,
	\end{equation}
	and $\tilde{d}_n\sim 2\sqrt{\log p}$, we have $P_2\to 0$.
	
	Next, we turn to $P_1$ which we write as follows
	\begin{align*}
		P_1&= \P\Big( \Vert\x_{i^*}-\x_{j^*} \Vert_2< \frac{-(1+\vep) g_{n,\alpha}}{2\Vert \mu_{i^*}-\mu_{j^*}\Vert_2} +\frac{\Vert \mu_{i^*}-\mu_{j^*}\Vert_2}{2}\Big)\,.
	\end{align*}
	Note that by \eqref{con} we have 
	$$\frac{-(1+\vep) g_{n,\alpha}}{2\Vert \mu_{i^*}-\mu_{j^*}\Vert_2} +\frac{\Vert \mu_{i^*}-\mu_{j^*}\Vert_2}{4}>0$$
	for sufficiently large $n$. Therefore, we obtain
	\begin{align}
		P_1&\ge \P\Big( \Vert\x_{i^*}-\x_{j^*} \Vert_2< \tfrac{1}{4} \Vert \mu_{i^*}-\mu_{j^*}\Vert_2\Big)\\
		&= \P\bigg( \frac{\Vert\x_{i^*}-\x_{j^*} \Vert_2^2-2n}{\sqrt{2n(\E[X^4]+1)}} <\frac{\tfrac{1}{16} \Vert \mu_{i^*}-\mu_{j^*}\Vert_2^2-2n}{\sqrt{2n(\E[X^4]+1)}}\bigg)\,.
	\end{align}
	In combination with \eqref{eq:clllt} and the fact that
	$$\lim_{\nto} \frac{\tfrac{1}{16} \Vert \mu_{i^*}-\mu_{j^*}\Vert_2^2-2n}{\sqrt{2n(\E[X^4]+1)}} =\infty$$
	we deduce $P_1\to 1$, as $\nto$.
\end{proof}

\subsection{Largest off-diagonal entry of a sample covariance matrix}\label{sec:3.3}

In the literature, the largest off-diagonal entry of a sample covariance matrix is a popular and powerful statistic for structural tests on the underlying dependence structure of a population; we refer to the review paper \cite{cai:2017} for an extensive summary and detailed references.   
Let $\y_1,\ldots,\y_n$ with $\y_k:=(Y_{k1},\ldots, Y_{kp})$ be a random sample from the multivariate normal population $\mathcal{N}_p(0,\Sigma_p)$, where for $\rho_n\in [0,1]$ the positive semidefinite population covariance matrix is
\begin{align*}
	\Sigma_p=(1-\rho_n)\bfI_p+\rho_n \mathbf{1}_p \mathbf{1}_p^{\top}\,
\end{align*}
with $\bfI_p$ the $p\times p$ identity matrix and $\mathbf{1}_p$ denotes the $p$-dimensional vector with all ones.
An important statistic for testing independence in high dimensions is $\sqrt{n}$ times the largest off-diagonal entry of the sample covariance matrix $n^{-1} \sum_{k=1}^n \y_k^{\top} \y_k$, which is given by
\begin{align*}
	W_n:=\max_{1\le i <j\le p} n^{-1/2}\sum_{k=1}^n Y_{ki}Y_{kj}\,.
\end{align*}
The study of $W_n$ was heavily influenced by Jiang.
In \cite{jiang:2004b}, assuming $\rho_n=0$ he showed that that $W_n$ is asymptotically Gumbel distributed. Note that in this case our Theorem~\ref{thm:summax} is applicable with the function $f(x,y)=xy$. Other cases for $\rho_n$ are more involved and allow for interesting phase transitions. The following result is the main result of Fan and Jiang \cite{fan:jiang:2019} whose proof is quite long and involved.  Using our techniques we provide a significantly shorter alternative proof.  

\begin{theorem}\label{thm:sample}
	For a nonnegative sequence $(\rho_n)_{n\ge 1}$ satisfying $\sup_{n\ge 1}\rho_n<1/2$ set
	$$\mu_n:=\sqrt{n}\rho_n+(1-\rho_n)d_n+2\sqrt{\rho_n (1-\rho_n)}\sqrt{2\log p}\,.$$ 
	If $p=p_n\to\infty$ and $p=\exp(o(n^{1/3}))$, the following statements hold as $n\to\infty$.
	\begin{itemize}
		\item[$(i)$] If $\rho_n\sqrt{\log p}\to 0$, then
		\begin{align*}
			d_n(W_n-\mu_n) \cid G\,, 
		\end{align*}
		where $G$ is standard Gumbel distributed.
		\item[$(ii)$] If $\rho_n\sqrt{\log p}\to \lambda\in(0,\infty)$, then
		\begin{align*}
			\frac{W_n-\mu_n}{\sqrt{2}\rho_n}\cid \xi+\tfrac{1}{\sqrt{8}\lambda}\, G,
		\end{align*}
		where $\xi\sim\mathcal{N}(0,1)$, $G$ is as in $(i)$ and $G$ is independent of $\xi$.
		\item[$(iii)$] If $\rho_n\sqrt{\log p}\to \infty$, then
		\begin{align*}
			\frac{W_n-\mu_n}{\sqrt{2}\rho_n}\cid \mathcal{N}(0,1).
		\end{align*}
	\end{itemize}
\end{theorem}
\begin{proof}
	Following \cite{fan:jiang:2019} we first derive a decomposition of $W_n-\mu_n$. To this end, let $\xi_k,\xi_{ki}; k,i=1,2,\ldots$ be independent standard normal random variables. 
	Defining $\rho_n'=1-\rho_n$, we have 
	\begin{align*}
		\big(\sqrt{\rho_n}\xi_k+\sqrt{\rho_n'}\xi_{ki},\quad 1\leq k\leq n, 1\leq i\leq p\big) \eid \big(Y_{ki},\quad 1\leq k\leq n, 1\leq i\leq p\big),
	\end{align*}
	and therefore we will assume that
	\begin{align*}
		Y_{ki}=\sqrt{\rho_n}\xi_k+\sqrt{\rho_n'}\xi_{ki},\quad\quad 1\leq k\leq n,\quad 1\leq i\leq p.
	\end{align*}
	Denote
	\begin{align*}
		A_n=\frac{1}{\sqrt{n}}\sum_{k=1}^n (\xi_k^2-1)\,,\quad B_{nij}=\frac{1}{\sqrt{n}}\sum_{k=1}^n\xi_{ki}\xi_{kj}\,,\quad C_{nij}=\frac{1}{\sqrt{n}}\sum_{k=1}^n \xi_k(\xi_{ki}+\xi_{kj})
	\end{align*}
	for all $1\leq i\leq j\leq p$. Setting $d_{n,2}=\sqrt{2\log p} -(\log\log p+\log 4\pi)/(2(2\log p)^{1/2})$ we get
	\begin{align}\label{eq:dec1}
		\frac{1}{\sqrt{n}}\sum_{k=1}^n Y_{ki}Y_{kj} -\mu_n=\rho_nA_n+\rho_n'(B_{nij}-d_n)+\sqrt{\rho_n\rho_n'}(C_{nij}- 2d_{n,2})+\sqrt{\rho_n\rho_n'} D_n\,,
	\end{align} 
	where
	\begin{equation*}
		D_n:=2 d_{n,2}-2\sqrt{2\log p}=O\Big(\tfrac{\log \log p}{\sqrt{\log p}}\Big)\,, \qquad\nto\,.
	\end{equation*}
	Next, we make the following claims about the terms in \eqref{eq:dec1}.
	\begin{itemize}
		\item[(1)] $A_n\cid \sqrt{2} \xi$, as $\nto$, where $\xi \sim N(0,1)$.
		\item[(2)] $d_n (\max_{1\le i<j\le p}B_{nij}-d_n) \cid G$, as $\nto$, where $G$ is standard Gumbel distributed. 
		\item[(3)] If $\rho_n\sqrt{\log p}\to 0$, then $d_n\sqrt{\rho_n\rho_n'} \,\max_{1\le i<j\le p}(C_{nij}- d_{n,2}) \cip 0$, as $\nto$. If\linebreak $\lim_{\nto} \rho_n\sqrt{\log p} >0$, then $\sqrt{\rho_n'/\rho_n} \,\max_{1\le i<j\le p}(C_{nij}- d_{n,2}) \cip 0$, as $\nto$.
	\end{itemize}
	Now we shall prove (1)-(3). (1) holds by the central limit theorem.
	Since $(\xi_{ki})$ are iid centered random variables with unit variance, (2) follows from Theorem \ref{thm:summax} choosing the function $f(x,y)=xy$ in \eqref{eq:additive}; see also \cite{li:qi:rosalsky:2012, heiny:mikosch:yslas:2021} for additional references where this result was derived.
	
	Regarding (3), we note that $\max_{1\le i<j\le p}C_{nij} = S_{(1)}+S_{(2)}$, where $S_{(1)}$ (respectively $S_{(2)}$) denotes the first (respectively second) largest order statistic of $n^{-1/2} \sum_{k=1}^n \xi_k \xi_{ki}, 1\le i\le p$. From Corollary 3.3 in \cite{heiny:mikosch:yslas:2021} and its proof we know that, as $\nto$,
	\begin{equation}\label{eq:dgsrr}
		d_{n,2}(S_{(1)}-d_{n,2},S_{(2)}-d_{n,2})\cid (-\log E_1, -\log (E_1+E_2))\,,
	\end{equation}
	where $E_1,E_2$ are independent, unit exponential random variables. It follows that, as $\nto$,
	\begin{equation}\label{eq:limitC}
		d_{n,2}\Big(\max_{1\le i<j\le p}C_{nij}-2d_{n,2}\Big) \cid -\log E_1 -\log (E_1+E_2)\,.
	\end{equation}
	Using $d_{n,1} \sim \sqrt{4 \log p}$ and $d_{n,2} \sim \sqrt{2 \log p}$, we deduce (3) from \eqref{eq:limitC}.\\
	\par
	
	In view of \eqref{eq:dec1}, we have 
	\begin{equation*}
		W_n-\mu_n=\rho_nA_n+\max_{1\le i<j\le p} \Big[\rho_n'(B_{nij}-d_n)+\sqrt{\rho_n\rho_n'}(C_{nij}- 2d_{n,2}) \Big] +O\Big(\tfrac{\sqrt{\rho_n\rho_n'} \log \log p}{\sqrt{\log p}}\Big)\,.
	\end{equation*}
	We now consider the three cases from the theorem.
	
	Case (i): $\rho_n\sqrt{\log p}\to 0$. By (1) and (3), an application of the Slutsky lemma yields that the sequence $d_n(W_n-\mu)$ has the same distributional limit as $\rho_n' d_n (\max_{1\le i<j\le p}B_{nij}-d_n)$. By (2) and since $\rho_n'\to 1$, it holds that  
	$$\rho_n' d_n (\max_{1\le i<j\le p}B_{nij}-d_n) \cid G\,, \qquad \nto\,.$$
	
	Case (ii): $\rho_n\sqrt{\log p}\to\lambda\in(0,\infty)$. By (3), the sequence $(W_n-\mu)/(\sqrt{2} \rho_n)$ has the same distributional limit as
	$$\frac{A_n}{\sqrt{2}}+\frac{\rho_n'}{\sqrt{2} \rho_n} \Big(\max_{1\le i<j\le p} B_{nij}-d_n\Big)\,,$$
	which is a sum of two independent terms. The first one converges to a standard normal random variable $\xi$ by (1). Noting that 
	$$\frac{\rho_n'}{\sqrt{2} \rho_n}\sim \frac{\sqrt{\log p}}{\sqrt{2} \lambda}\sim \frac{1}{\sqrt{8}\lambda} d_n\,,\qquad \nto\,,$$
	the second term converges to $(\sqrt{8}\lambda)^{-1} G$, where $G$ is standard Gumbel distributed.
	
	Case (iii): $\rho_n\sqrt{\log p}\to\infty$. Noting that $d_n/\rho_n\to 0$, an application of the Slutsky lemma combined with (2) and (3) shows that the sequence $(W_n-\mu)/(\sqrt{2} \rho_n)$ has the same distributional limit as $A_n/\sqrt{2}$, which converges to a standard normal random variable by (1).
	
	The proof of the theorem is complete.
\end{proof}

\section{Extension to point process convergence}\label{sec:pointproc}

In Theorem \ref{thm:summax} we considered the asymptotic behavior of the maximum of dependent and identically distributed random variables $(T_{ij})$. 
In the case of an iid sequence of real-valued random variables $(Y_i)$ the convergence  
\begin{align*}
	\lim_{\nto} \P\bigg(\frac{1}{a_n}\Big(\max_{1\leq i\leq n} Y_i-b_n\Big)\le x \bigg)= H(x), \qquad x\in \R,
\end{align*}
to some max-stable distribution function $H$ and the weak point process convergence 
\begin{align*}
	\sum_{i=1}^n\vep_{(Y_i-b_n)/a_n}\cid N,
\end{align*}
where $N$ is a Poisson random measure with mean measure $\mu_H(a,b)=\log H(b)-\log H(a)$ for $a<b\in\operatorname{supp}(H)$, are equivalent (see \cite[Proposition 3.21]{resnick:1987}). Here $\vep_x(A)=\1_{\{x\in A\}}$ for a set $A\subset \R$. 

A sequence of point processes $(N_n)_n$ on the state space $E\subset\overline{\R}$ equipped with the $\sigma$-algebra $\mathcal{E}$ of the Borel sets converges weakly in the space of all point measures on $E$ to a point process $N$ on $E$, if for any bounded Borel sets $A_1,\ldots, A_m\in\mathcal{E}$ with $\P(N(\partial A_i)=0)=1$, $i=1,\ldots,m$ and $m\geq 1$, where $\partial A$ denotes the boundary of a set $A$, it holds that
\begin{align*}
	\P(N_n(A_1),\ldots,N_n(A_m))\to\P(N(A_1),\ldots,N(A_m)),
\end{align*}
as $n\to\infty$; see \cite[Definition 5.2.1]{embrechts:kluppelberg:mikosch:1997}.  

Thus, in case of iid points $(Y_i)$, the convergence of the maximum is equivalent to the convergence of the point processes. In general, the latter is a stronger statement.
In our case the random variables $(T_{ij})$ in \eqref{eq:defTij} have a special dependency structure so that we cannot directly conclude the convergence of the point processes
\begin{align*}
	N_n:=\sum_{1\leq i<j\leq p}\vep_{d_n(T_{ij}-d_n)}
\end{align*}
from the convergence in distribution of $\max_{1\le i<j\le p} d_n(T_{ij}-d_n)$ to standard Gumbel, which was established in Theorem \ref{thm:summax}. Nevertheless, with some additional effort we obtain convergence of $(N_n)$. 
\begin{theorem}\label{thm:summaxpp}
	In the setting of Theorem \ref{thm:summax} it holds that
	\begin{equation}\label{eq:maxfpp}
		N_n\cid N\,, \qquad \nto\,,
	\end{equation}
	where $N$ is a Poisson random measure with mean measure $\mu(x,\infty)=\e^{-x}$ for $x\in\R$.
\end{theorem}

\begin{proof}
	Since $\mu$ has a density, the limit process $N$ is simple and we can apply Kallenberg's Theorem (see for instance \cite[p.233, Theorem 5.2.2]{embrechts:kluppelberg:mikosch:1997} or \cite[p.35, Theorem 4.7]{kallenberg:1983}). Therefore, it suffices to prove that for any finite union of bounded 
		intervals 
	\begin{align*}
		R=\bigcup_{k=1}^q A_k\subset \R,\qquad\text{with}\qquad A_k=(r_k,s_k],
	\end{align*}
	it holds that 
	\begin{equation}\label{eq:essessbr}
		\lim_{n\to\infty} \E[N_n(R)]=\mu(R) \qquad \text{ and }\qquad 
		\lim_{n\to\infty} \P(N_n(R)=0)=\e^{-\mu(R)}\,.
	\end{equation}
	Without loss of generality we assume that the $A_k$'s are disjoint. We start with the first limit in \eqref{eq:essessbr} and get
	\begin{align*}
		\E[N_n(R)]=\sum_{k=1}^q\E[N_n(A_k)]=\sum_{k=1}^q \frac{p(p-1)}{2}\P(T_{12}\in A_k). 
	\end{align*}
	According to assertion (A1) in the proof of Theorem \ref{thm:summax} we have the convergence 
	\begin{align*}
		\sum_{k=1}^q \frac{p(p-1)}{2}\P(T_{12}\in A_k)\to \sum_{k=1}^q\mu(A_k)=\mu(R),\qquad n\to\infty.  
	\end{align*}
	Regarding the second limit in \eqref{eq:essessbr} we see that
	\begin{align*}
		|\P(N_n(R)=0)-\e^{-\mu(R)}| \leq |\P(N_n(R)=0)-\e^{-\E[N_n(R)]}|+|\e^{-\E[N_n(R)]}-\e^{-\mu(R)}|\,,  
	\end{align*}
	where the latter term tends to zero as $\E[N_n(R)]\to \mu(R)$. To show that the first term converges to zero as well we apply Lemma \ref{lem:poissonapprox2}. To this end, set $I=I^{(n)}= \{(i,j): 1\le i<j\le p\}$. For $\alpha=(i,j) \in I$, we define $W=N_n(R)$ and 
	\begin{equation*}
		B_{\alpha}= \{(k,l): (k,l)\in I \text{ and } |\{i,j,k,l\}| \leq 3\}\,.
	\end{equation*}
	Then we see that 
	\begin{align*}
		|\P(N_n(R)=0)-\e^{-\E[N_n(R)]}|\leq (1\wedge \E[N_n(R)]^{-1})(b_1+b_2+b_3),
	\end{align*}
	Here $\E[N_n(R)]\to \mu(R)$ and the terms $b_1,b_2$ and $b_3$ are as in the proof of Theorem~\ref{thm:summax}, where it is shown that those three terms tend to zero.
\end{proof}

Since Theorem \ref{thm:interpoint} follows from Theorem \ref{thm:summax}, we can also formulate a point process convergence result for the interpoint distances $D_{ij}^{(2)}$ defined in \eqref{eq:defDij}. 

\begin{corollary}
	Under the assumptions of Theorem \ref{thm:interpoint} it holds that 
	\begin{align*}
		\sum_{1\leq i<j\leq p}\vep_{c_n^{(2)} (D_{ij}^{(2)}-b_n^{(2)})}\cid N,\qquad n\to\infty,
	\end{align*}
		where $b_n^{(2)}$ and $c_n^{(2)}$ are defined in \eqref{normconst}, and $N$ is a Poisson random measure with mean measure $\mu(x,\infty)=\e^{-x}$ for $x\in\R$.
\end{corollary}

An advantage of having even point process convergence is that the convergence of the joint distribution of a fixed number of upper order statistics is a direct consequence. 
\begin{corollary}
	Assume the conditions of Theorem \ref{thm:summax}. 
	For $k\in \N$ let 
	$$G_{n,(1)}\ge G_{n,(2)}\ge \cdots G_{n,(k)}$$
	be the $k$ largest upper order statistics of the random variables $(d_n(T_{ij}-d_n))$, $1\leq i<j\leq p$. Then for real numbers $x_k<\ldots<x_1$ the distribution function
	\begin{align*}
		\P(G_{n,(1)}\leq x_1,\ldots,G_{n,(k)}\leq x_k)
	\end{align*}
	converges as $n\to\infty$  to 
	\begin{align}\label{eq:sedtgses}
		\P\Big(N(x_1,\infty)= 0,\ldots,N(x_k,\infty)\leq k-1\Big)=\P(-\log \Gamma_1\leq x_1,\ldots, -\log \Gamma_k\leq x_k),
	\end{align}
	where $N$ is a Poisson random measure with mean measure $\mu(x,\infty)=\e^{-x}$ and $\Gamma_i=E_1+\ldots+E_i$ for iid standard exponentially distributed random variables $(E_i)_{i\ge 1}$.
\end{corollary}
Note that in \eqref{eq:sedtgses} we have implicitly used the representation $N\eid \sum_{i=1}^{\infty} \vep_{-\log \Gamma_i}$; we refer to \cite{resnick:1987} for details.

\section{Proof of Theorem \ref{thm:summax}}\label{sec:mainproof}\setcounter{equation}{0}

{\it In the following  $C$, $c_1$, $c_2$, $c_3$ and $c_4$ are positive constants that do not depend on $n$ and that may vary from line to line.}

\subsection{Preliminaries}\label{prel}

We claim that
\begin{equation}\label{eq:neu_maxf}
		\max_{1\le i<j\le p} d_n(T_{ij}-d_n)\cid G, \qquad \nto\,,
	\end{equation}
follows from the assertions
\begin{itemize}
	\item[(A1)] $\tilde{p}\, \P\Big( d_n\big(\frac{1}{\sqrt{n}}\sum_{l=1}^n Z_{1,2,l}-d_n\big)>x\Big) \to \e^{-x}$ and
	\item[(A2)] $\P\Big( d_n\big(\frac{1}{\sqrt{n}}\sum_{l=1}^n Z_{1,2,l}-d_n\big)>x, d_n\big(\frac{1}{\sqrt{n}}\sum_{l=1}^n Z_{1,3,l}-d_n\big)>x\Big)=o(p^{-3})$,
\end{itemize}
where $x\in \R$ and $\nto$.\\
Our goal is to prove this claim by means of Lemma \ref{lem:poissonapprox}. To this end, set $I=I^{(n)}= \{(i,j): 1\le i<j\le p\}$. For $\alpha=(i,j) \in I$, we define $\eta_{\alpha}=  d_n\big(\frac{1}{\sqrt{n}}\sum_{l=1}^n Z_{i,j,l}-d_n\big)$ and 
\begin{equation*}
	B_{\alpha}= \{(k,l)\in I \,:\, |\{i,j,k,l\}| \leq 3\}\,.
\end{equation*}
For $x\in \R$, we have by (A1) that
\begin{equation*}
	\begin{split}
		\lambda &:=\lambda^{(n)}:=\sum_{\alpha\in I} \P( \eta_{\alpha} > x) \to \e^{-x}\,, \qquad \nto\,.
	\end{split}
\end{equation*}
Recall the definition of $b_1,b_2,b_3$ from Lemma \ref{lem:poissonapprox}. By construction, $\eta_{\alpha}$ is independent of $\{\eta_{\beta}:\beta\notin B_{\alpha}\}$ for all $\alpha \in I$. Thus, $b_3=0$. Therefore, it remains to prove that $b_1, b_2\to 0$ as $\nto$.
We start with $b_1$. We easily see that $|B_{\alpha}| \le 2p$ which implies 
\begin{equation}\label{eq:b1}
	\begin{split}
		b_1 &\le \frac{p^2}{2} 2p \Big( \P( \eta_{12} > x) \Big)^2 = \frac{4}{p} \Big[ \underbrace{\frac{p^2}{2} \P( \eta_{12} > x)}_{\to e^{-x}} \Big]^2 = O(1/p)\,, \qquad \nto\,.
	\end{split}
\end{equation}
Regarding $b_2$, we conclude from (A2) that
$$b_2 \le p^3 \, \P( \eta_{12} > x,\eta_{13} > x)\to 0\,,\qquad \nto\,,$$ 
which establishes the claim in view of Lemma~\ref{lem:poissonapprox}.


\subsection{Proof under condition {\rm (C1)}}\label{sec:proofC1}
First, we will show that we may work with the truncated and recentered random variables
\begin{align*}
	\hat{Z}_{i,j,l}:=Z_{i,j,l}\mathds{1}_{\lbrace|Z_{i,j,l}|\leq \tau_n\rbrace}-\E[Z\mathds{1}_{\lbrace|Z|\leq \tau_n\rbrace}],
\end{align*}
where $\tau_n:=n^{1/s}$.
By the Slutsky lemma, \eqref{eq:neu_maxf} 
is an immediate consequence of  
\begin{align}\label{slutsky1}
	\lim_{\nto} \P\Big(\max_{1\le i<j\le p} \frac{1}{\sqrt n} \sum_{l=1}^n \hat{Z}_{i,j,l} \leq x/d_n+d_n \Big)= \exp(-\e^{-x})\,, \qquad x\in \R 
\end{align}
and
\begin{align}\label{slutsky2}
	\frac{d_n}{\sqrt n}  \max_{1\le i<j\le p} \Big| \sum_{l=1}^n \big(Z_{i,j,l}- \hat{Z}_{i,j,l}\big)\Big|\cip 0\,,\qquad \nto\,.
\end{align}
Regarding \eqref{slutsky2}, we get by the Fuk-Nagaev inequality  \cite[p.78]{petrov:1995} for $\vep>0$ that
\begin{align}
	&\P\Big(\frac{d_n}{\sqrt n}  \max_{1\le i<j\le p} \Big| \sum_{l=1}^n \big(Z_{i,j,l}- \hat{Z}_{i,j,l}\big)\Big|>\vep \Big) \le \tilde{p}\, \P\Big(  \Big| \sum_{l=1}^n \big(Z_{1,2,l}- \hat{Z}_{1,2,l}\big)\Big|>\frac{\sqrt n}{d_n}\vep \Big)\nonumber\\
	&\lesssim \tilde{p}\,n\,\E[|Z-\hat{Z}|^s]\big(\frac{\sqrt{n}\vep}{d_n}\big)^{-s}+\tilde{p}\,\exp\Big(-c_2\frac{\vep^2}{d_n^2\Var(Z-\hat{Z})}\Big),\label{fukn}
\end{align}
where $\hat{Z}=Z\mathds{1}_{\lbrace|Z|\leq \tau_n\rbrace}-\E[Z\mathds{1}_{\lbrace|Z|\leq \tau_n\rbrace}]$. Since $\E[|Z|^s(\log|Z|)^{s/2}]<\infty$ we get
\begin{align*}
	\E[|Z-\hat{Z}|^s]&\lesssim\,\E[|Z|^s\1_{\{|Z|>\tau_n\}}] \\
	&\leq \frac{\E[|Z|^s(\log|Z|)^{s/2}\1_{\{|Z|>\tau_n\}}]}{(\log(\tau_n))^{s/2}}=o((\log n)^{-s/2}), \qquad n\to\infty.
\end{align*}
Therefore, the first term in \eqref{fukn} tends to zero as $n\to \infty$ for $p=O(n^{(s-2)/4})$. Additionally, we obtain
\begin{align*}
	\Var(Z-\hat{Z})&=\Var(Z\1_{\{|Z|>\tau_n\}})\le \E[Z^2\1_{\{|Z|>\tau_n\}}]\\
	&\leq \frac{\E[|Z|^s(\log|Z|)^{s/2}]}{\tau_n^{s-2}(\log(\tau_n))^{s/2}}\lesssim n^{-(s-2)/s}(\log n)^{-s/2}.
\end{align*} 
Hence, as $n\to\infty$, the second term in \eqref{fukn} tends to zero for  $p=O(n^{(s-2)/4})$. This establishes \eqref{slutsky2}.  

Therefore, to complete the proof of Theorem \ref{thm:summax} 
under condition {\rm (C1)} it suffices to prove \eqref{slutsky1}.
To this end, we will verify conditions (A1), (A2) with $Z$ replaced by $\hat{Z}$.
Note that $|\hat{Z}|$ is bounded from above by $\tau_n+\big|\E[Z\mathds{1}_{\lbrace|Z|\leq \tau_n\rbrace}]\big|=\tau_n+o(1)$. For simplicity\footnote{All our arguments would remain valid if $|\hat{Z}|\le 2\tau_n$.} we will assume that $|\hat{Z}|\le \tau_n$.

Let $x\in \R$ and set $x_n:=x/d_n+d_n$ and $\varepsilon_n:=(\log p)^{-1}$. By Theorem 1.1 of \cite{Zaitsev:1987}, we get for a normal random variable $N\sim\mathcal{N}(0,\Var(\hat{Z}))$ that 
\begin{align}\label{zai1}
	\P \Big( \frac{1}{\sqrt n} \sum_{l=1}^n \hat{Z}_{1,2,l} >x_n\Big) &\ge \P( N>x_n+\varepsilon_n)-c_1\exp\Big(-c_2\frac{\sqrt{n}\varepsilon_n}{\tau_n}\Big)\,,\\
	\P \Big( \frac{1}{\sqrt n} \sum_{l=1}^n \hat{Z}_{1,2,l} >x_n\Big) &\le \P( N>x_n-\varepsilon_n)+ c_3\exp\Big(-c_4\frac{\sqrt{n}\varepsilon_n}{\tau_n}\Big)\,.\label{zai2}
\end{align}
For the exponential term we get
\begin{align}\label{exp}
	p^2\exp\Big(-c_2\frac{\sqrt{n}\varepsilon_n}{\tau_n}\Big)=\exp\Big(2\log p-c_2\, n^\frac{s-2}{2s} (\log p)^{-1} \Big),
\end{align}
which tends to zero for $p=O(n^{\frac{s-2}{4}}$).
Since $\Var(\hat{Z})\le \Var(Z)=1$, we have
\begin{align*}
	\P( N>x_n-\varepsilon_n)\leq\P\Big( \tfrac{N}{\sqrt{\Var(N)}}>x_n-\varepsilon_n\Big)=\ov{\Phi}(x_n-\vep_n).
\end{align*}
By Mill's ratio and the fact that $(x_n-\varepsilon_n)^2=x_n^2 +o(1)$, one obtains 
\begin{align}\label{topro}
	\frac{p^2}{2} \,\ov\Phi(x_n-\vep_n)&\sim \frac{p^2}{2} \,\frac{1}{\sqrt{2\pi}(x_n-\vep_n)}\exp\Big(-\frac{(x_n-\vep_n)^2}{2}\Big)\nonumber\\
	&\sim \frac{p^2}{2} \,\frac{1}{\sqrt{2\pi}x_n}\exp\Big(-\frac{x_n^2}{2}\Big)\to \e^{-x}\,, \qquad \nto.
\end{align}
From \eqref{exp} and \eqref{topro} we deduce that $p^2/2$ times the \rhs~in \eqref{zai2}  converges to $\e^{-x}$. A similar argument yields that $p^2/2$ times the \rhs~in \eqref{zai1} converges to $\e^{-x}$. A combination of the last two observations proves that
$$ \lim_{\nto} \frac{p^2}{2} \P \Big( \frac{1}{\sqrt n} \sum_{l=1}^n \hat{Z}_{1,2,l} >x_n\Big)= \e^{-x}\,, \qquad x\in \R\,.$$

Next, we turn to the proof of (A2) with $Z$ replaced by $\hat{Z}$, that is 
\begin{align}\label{eq:A2Zhat}
	\lim_{\nto} p^3\, \P\Big(\min_{j\in \{2,3\}} \frac{1}{\sqrt n} \sum_{l=1}^n \hat{Z}_{1,j,l} > x/d_n+d_n\Big)&=0,\qquad\, x\in \R\,.
\end{align}
We set 
$\lambda_n:=(\log p)^{-1/2}$ and $\rho_n:=\Cov(\hat{Z}_{1,2,1},\hat{Z}_{1,3,1})$.
For a normal distributed vector
\begin{equation}\label{eq:N1N2}
	\begin{pmatrix} N_1 \\ N_2 \end{pmatrix}
	\sim \mathcal{N} \left( \begin{pmatrix} 0 \\ 0 \end{pmatrix}, 
	\begin{pmatrix}
		\Var(\hat{Z}) & \rho_n\\
		\rho_n & \Var(\hat{Z})
	\end{pmatrix} \right)
\end{equation}
Theorem 1.1 of \cite{Zaitsev:1987} yields
\begin{align*}
	\P \Big(\min_{j\in \{2,3\}} \frac{1}{\sqrt n} \sum_{l=1}^n \hat{Z}_{1,j,l} >x_n\Big)\leq \P\Big(\min_{j\in \{1,2\}} N_j>x_n-\lambda_n\Big)+ c_1\exp\Big(-c_2\frac{\sqrt{n}\lambda_n}{\tau_n}\Big).
\end{align*}
Since $p= O( n^{\frac{s-2}{4}})$, the exponential term
\begin{align*}
	p^3\,\exp\Big(-c_2\frac{\sqrt{n}\lambda_n}{\tau_n}\Big)=\exp\Big(3\log p-c_2\,n^{\frac{s-2}{2s}}(\log p)^{-1/2}\Big)
\end{align*}
tends to zero as $n\to\infty$.
Furthermore, by virtue of Lemma \ref{lem:minN}, we have
\begin{align*}
	p^3\, \P\Big(\min_{j\in \{1,2\}} N_j>x/d_n+d_n-\lambda_n \Big)&=o(1),\quad\,\, n\to\infty\,,
\end{align*}
completing the proof of \eqref{eq:A2Zhat}.

\subsection{Proof under condition {\rm (C2)}}
In this setting, we assumed $p=\exp(o(n^{r/(2-r)}))$ for some $r\in (0,2/3]$. First, we consider the case $r\in(0,1/2]$. Recalling the notation $x_n:=x/d_n+d_n$, an application of part (ii) of Lemma \ref{lin} yields 
\begin{align}\label{eq:sesegsdd2}
	&\frac{p^2}{2} \P\Big( \frac{1}{\sqrt n} \sum_{l=1}^n Z_{1,2,l} > x_n \Big)\sim \frac{p^2}{2} \ov \Phi(x_n)\to \e^{-x} \,, \qquad x\in \R\,,\nto\,.
\end{align}
It remains to show (A2). To this end, we bound
\begin{align*}
	& \P\Big(\min_{j\in \{2,3\}} \frac{1}{\sqrt n} \sum_{l=1}^n Z_{1,j,l} > x_n\Big)\le 
	\P\Big(\frac{1}{\sqrt n} \sum_{l=1}^n (Z_{1,2,l}+Z_{1,2,l}) > 2x_n\Big)\\
	&=  \P\Big(\frac{1}{\sqrt{ 2(1+\rho) n}} \sum_{l=1}^n (Z_{1,2,l}+Z_{1,3,l}) > \frac{\sqrt{2}x_n}{\sqrt{1+\rho}}\Big) \\
	&\leq\exp\Big(-\frac{1-\varepsilon}{1+\rho}x_n^2\Big), \qquad \nto\,,
\end{align*}
for $\varepsilon>0$, where part (i) of Lemma \ref{lin} was applied in the last line. Note that $x_n=o(n^{r/(2(2-r))})$ under {(C2)}. Choosing  $\varepsilon=\frac{1}{2}\big(\frac{1}{3}-\rho\big)>0$ and using the fact that $x_n^2\sim 4 \log p$, it follows 
\begin{align*}
	& p^3 \exp\Big(-\frac{1-\varepsilon}{1+\rho}x_n^2\Big)
	=\exp \Big(3\log p-\frac{1-\varepsilon}{1+\rho}x_n^2 \Big)\to 0\,,\qquad x\in \R\,, \nto\,,
\end{align*}
which finishes the proof of (A2) for $r\in(0,2/3]$. \\

Therefore it remains to show (A1) in the case $r\in(1/2,2/3]$, our strategy is to ultimately apply part (iii) of Lemma \ref{lin}. However, since the moment generating function of $|Z|$ is not necessarily finite in some neighborhood of zero, we will work with the truncated random variables
\begin{align*}
	\hat{Z}_{i,j,l}:=Z_{i,j,l}\1_{\{|Z_{i,j,l}|\leq\tau_n\}}\quad\text{and}\,\,\,\tau_n:={\eta^{-1/r}n^{1/(2-r)}}\,,
\end{align*}
with the constant $\eta>0$ from (C2).
A union bound and the Markov inequality show that
\begin{align*}
	&\P\Big(\frac{d_{n}}{\sqrt{n}}\max_{1\leq i<j\leq p}\Big|\sum_{l=1}^n (Z_{i,j,l}-\hat{Z}_{i,j,l})\Big|>\varepsilon\Big)
	\leq p^2\,\P\Big(\max_{1\leq l\leq n}|Z_{1,2,l}|>\tau_n\Big)\\
	&\leq p^2n\,\E[\exp(\eta |Z|^r)]\exp(-n^{r/(2-r)})=o(1),\qquad\,n\to\infty.
\end{align*}
Therefore, (A1) is implied by 
\begin{align}\label{eq:sdess}
	\frac{p^2}{2} \P\Big( \frac{1}{\sqrt n} \sum_{l=1}^n \hat{Z}_{1,2,l} > x_n \Big)\to \e^{-x} \,, \qquad \nto\,,
\end{align}
where for $x\in \R$ we write $x_n=x/d_{n} +d_{n}$ and recall that, under (C2) with $r>1/2$, $d_n=d_{n}^{(\E[Z^3])}$ (see \eqref{dn1}
for the latter's definition). 
Now we turn to the proof of \eqref{eq:sdess} and get
\begin{align*}
	& \P\Big( \frac{1}{\sqrt n} \sum_{l=1}^n \hat{Z}_{1,2,l} > x_n \Big)=\P\Big( \frac{x_n}{\sqrt n} \sum_{l=1}^n \hat{Z}_{1,2,l}-x_n^2 > 0 \Big)= \P\Big( \sum_{l=1}^n \xi_l > 0 \Big),
\end{align*}
where
\begin{align*}
	\xi_l:=\frac{x_n}{\sqrt{n}}\hat{Z}_{1,2,l}-\frac{x_n^2}{n}\leq\frac{x_n}{\sqrt{n}}\tau_n^{1-r}|Z_{1,2,l}|^r=o(1)|Z_{1,2,l}|^r\,, \qquad \nto\,.
\end{align*}
Under (C2) we have $\sup_{1\leq l\leq n,{}n\geq 1}\E[\e^{\xi_l}]<\infty$. For $1\le l\le n$ we have
\begin{align*}
	\E[\xi_l]&=\frac{x_n}{\sqrt{n}}\E[Z\1_{\{|Z|\leq\tau_n\}}]-\frac{x_n^2}{n}= -\frac{x_n^2}{n}+O\big(\frac{x_n}{\sqrt{n}}\exp(-1/2\,n^{r/2-r)})\big) \quad \text{ and } \\
	\Var(\xi_l)&=\frac{x_n^2}{n}\Var(Z\1_{\{|Z|\leq\tau_n\}})=\frac{x_n^2}{n}(1+O(\exp(-1/2\,n^{r/2-r)}))).
\end{align*}
We set $\mu_n:=n\E[\xi_1]$, $\sigma_n^2:=n\Var(\xi_1)$ and  obtain, as $n\to\infty$
\begin{align}
	-\frac{\mu_n}{\sigma_n}&=\frac{x_n+O(\sqrt{n}\exp(-1/2\,n^{r/(2-r)}))}{\sqrt{1+O(\exp(-1/2\,n^{r/(2-r)}))}}\nonumber\\
	&= x_n+ \Big(\frac{1}{\sqrt{1+O(\exp(-1/2\,n^{r/(2-r)}))}}-1\Big)x_n+O(\sqrt{n}\exp(-1/2\,n^{r/(2-r)}))\nonumber\\
	&= x_n+O\Big(x_n\exp(-1/4\,n^{r/(2-r)})\Big).\label{argu}
\end{align}
Additionally, we get, as $n\to\infty$
\begin{align*}
	\frac{\E\big[(\xi_l-\E[\xi_l])^3\big]}{\Var^{3/2}(\xi_l)}&=  \frac{\E\big[(\hat{Z}_{1,2,l}-\E[Z\1_{\{|Z|\leq \tau_n\}}])^3\big]}{(1+O(\exp(-1/2\,n^{r/2-r)})))^{3/2}}\\
	&=\frac{\E[Z^3]+O(\exp(-1/2\,n^{r/2-r)}))}{(1+O(\exp(-1/2\,n^{r/2-r)})))^{3/2}}=\E[Z^3]+O\Big(\exp(-1/4\,n^{r/(2-r)})\Big),
\end{align*}
where we used similar arguments for the last step as for \eqref{argu}.
Notice that $\mu_n/\sigma_n\sim x_n= o(n^{r/(2(2-r))})$ which implies $\mu_n/\sigma_n=o(n^{1/4})$ for $r\in(1/2,2/3]$. Therefore, from part (iii) of Lemma \ref{lin} and the computations above, we deduce
\begin{align*}
	\P\Big( \sum_{l=1}^n \xi_l > 0 \Big)&=\P\Big( \frac{\sum_{l=1}^n (\xi_l-\E[\xi_l])}{\sigma_n} > -\frac{\mu_n}{\sigma_n} \Big)\\
	&\sim \ov \Phi\Big(-\frac{\mu_n}{\sigma_n}\Big)\,\exp\Big(\frac{(-\mu_n/\sigma_n)^3}{6n^{1/2}}\E\Big[\Big(\frac{\xi_1-\E[\xi_1]}{\sqrt{\Var(\xi_l)}}\Big)^3\Big]\Big)\\
	&\sim \ov \Phi(x_n) \,\exp\Big(\frac{x_n^3\E[Z^3]}{6n^{1/2}}\Big),\qquad n\to\infty.
\end{align*}
Using the definition of $d_n$, we get $x_n^3/n^{1/2}=\big(2\log\big(\tilde{p}\,\big)\big)^{3/2}/n^{1/2}+o(1)$. Note that this term does not necessarily tend to zero for $p=\exp(o(n^{r/(2-r)})$ with $r>1/2$. Hence as $\nto$, 
\begin{align*}
	&\tilde{p}\, \ov\Phi(x_n) \exp\Big(\frac{x_n^3\E[Z^3]}{6n^{1/2}}\Big)\sim \tilde{p}\, \ov \Phi(x_n)\exp\Big(\frac{\big(2\log\big(\tilde{p}\big)\big)^{3/2}\E[Z^3]}{6n^{1/2}}\Big)\to \e^{-x} \,, 
\end{align*}
where Mill's ratio was used in the last step.
That completes the proof of (A1) under (C2) and $r\in (1/2,2/3]$.

\subsection{Proof under condition {\rm (C3)}}
For the proof of Theorem \ref{thm:summax} 
under condition (C3), we will proceed similarly as under (C1).
We will show that we may truncate the random variables $Z_{i,j,l}$ and than we will verify conditions (A1), (A2) for the truncated variables.\\
We set
\begin{align*}
	\hat{Z}_{i,j,l}:=Z_{i,j,l}\mathds{1}_{\lbrace|Z_{i,j,l}|\leq \tau_n\rbrace}-\E[Z\mathds{1}_{\lbrace|Z|\leq \tau_n\rbrace}],
\end{align*}
where $\tau_n:=\big(\frac{2}{\eta}\log(np)\big)^{1/r}$.
As in the proof under (C1), it suffices to show that 
\begin{align}\label{slutsky1.2}
	\lim_{\nto} \P\Big(\max_{1\le i<j\le p} \frac{1}{\sqrt n} \sum_{l=1}^n \hat{Z}_{i,j,l} \leq x/d_n+d_n \Big)= \exp(-\e^{-x})\,, \qquad x\in \R\,,
\end{align}
and
\begin{align}
	\label{slutsky2.2} \frac{d_n}{\sqrt n}  \max_{1\le i<j\le p} \Big| \sum_{l=1}^n \big(Z_{i,j,l}- \hat{Z}_{i,j,l}\big)\Big|\cip 0.
\end{align}
First, we prove \eqref{slutsky2.2}. For $\vep>0$ we have
\begin{align}
	&\P\Big(\frac{d_n}{\sqrt n}  \max_{1\le i<j\le p} \Big| \sum_{l=1}^n \big(Z_{i,j,l}- \hat{Z}_{i,j,l}\big)\Big|>\vep \Big)
	\le \tilde{p}\, \P\Big(\frac{d_n}{\sqrt n}  \Big| \sum_{l=1}^n \big(Z_{1,2,l}- \hat{Z}_{1,2,l}\big) \Big|>\vep \Big)\nonumber\\
	&= \tilde{p}\, \P\Big(\frac{d_n}{\sqrt n}  \Big| \sum_{l=1}^n \big(Z_{1,2,l}\mathds{1}_{\lbrace|Z_{1,2,l}|> \tau_n\rbrace}-\E[Z\mathds{1}_{\lbrace|Z|> \tau_n\rbrace}]\big) \Big|>\vep \Big)\nonumber\\
	&\leq \tilde{p}\, \P\Big(\frac{d_n}{\sqrt n}  \Big| \sum_{l=1}^n Z_{1,2,l}\mathds{1}_{\lbrace|Z_{1,2,l}|> \tau_n\rbrace}\Big|>\vep/2 \Big)+ \tilde{p}\, \P\Big(\frac{d_n}{\sqrt n}  \Big| \sum_{l=1}^n \E[Z\mathds{1}_{\lbrace|Z|> \tau_n\rbrace}]\Big|>\vep/2 \Big).\label{2bterm}
\end{align}
Using condition (C3), we obtain
\begin{align*}
	\frac{1}{\sqrt n} \sum_{l=1}^n\E[|Z|\mathds{1}_{\lbrace|Z|> \tau_n\rbrace}]
	\leq \frac{\sqrt{n}}{\exp(\frac{\eta}{2}\tau_n^r)}\E[|Z|\exp(\frac{\eta}{2} |Z|^r)]
	\lesssim \frac{1}{\sqrt{n}p}\,.
\end{align*}
By virtue of $d_n/p \to 0$, the second term of \eqref{2bterm} tends to zero as $\nto$. 
Using the union bound and Markov's inequality, the first term of \eqref{2bterm} can be bounded by 
\begin{align*}
	p^2\, \P\Big(\max_{1\leq l\leq n}|Z_{1,2,l}|>\tau_n\Big)
	\leq p^2 n\, \P\Big(\exp(\eta |Z|^r)>\exp(\eta \tau_n^r)\Big)
	\lesssim \frac{p^2n}{\exp(\eta \tau_n^r)}=o(1),
\end{align*}
as $n\to\infty$.
This establishes \eqref{slutsky2.2}. 
To show \eqref{slutsky1.2}, we will verify conditions (A1), (A2) with $Z$ replaced by $\hat{Z}$.
We write $x_n:=x/d_n+d_n$. An application of Lemma \ref{lin} (iii) yields 
\begin{align}\label{eq:sesegsdd.2}
	&\frac{p^2}{2} \P\Big( \frac{1}{\sqrt n} \sum_{l=1}^n \hat{Z}_{1,2,l} > x_n \Big)\sim \frac{p^2}{2} \ov \Phi(x_n)\to \e^{-x} \,, \qquad \nto\,,
\end{align}
since $x_n\sim \sqrt{2\log p}=o(n^{1/6})$ as $p=\exp(o(n^{1/(3+2/r)}))$, which proves condition (A1).

It remains to show (A2), that is 
\begin{align}\label{eq:dsdgfd}
	p^3\, \P\Big(\min_{j\in \{2,3\}} \frac{1}{\sqrt n} \sum_{l=1}^n \hat{Z}_{1,j,l} > x/d_n+d_n\Big)&=o(1),\qquad n\to\infty.
\end{align}
Setting 
$\lambda_n:=(\log p)^{-1/2}$ and $\rho_n:=\Cov(\hat{Z}_{1,2,1},\hat{Z}_{1,3,1})$, one obtains as in Section~\ref{sec:proofC1} that
\begin{align*}
	&\P \Big(\min_{j\in \{2,3\}} \frac{1}{\sqrt n} \sum_{l=1}^n \hat{Z}_{1,j,l} >x_n\Big)\leq \P\Big(\min_{j\in \{1,2\}} N_j>x_n-\lambda_n\Big)+ c_1\exp\Big(-c_2\frac{\sqrt{n}\lambda_n}{\tau_n}\Big). 
\end{align*}
where the normal random variables $N_1,N_2$ satisfy \eqref{eq:N1N2}.
For $p= \exp(o( n^{(3+\frac{2}{r})^{-1}}))$, the exponential term
{\small\begin{align*}
		p^3\exp\Big(-c_2\frac{\sqrt{n}\lambda_n}{\tau_n}\Big)=\exp\Big(3\log p-c_2\frac{n^{1/2}}{(\log (np))^{1/r}\sqrt{\log p}}\Big)\leq \exp\Big(3\log p-c_2n^{(3+\frac{2}{r})^{-1}}\Big),
\end{align*}}
tends to zero, as $n\to\infty$, and so does $p^3\, \P\big(\min_{j\in \{1,2\}} N_j>x_n-\lambda_n \big)$ by Lemma \ref{lem:minN}, finishing the proof of \eqref{eq:dsdgfd}.

\subsection{Proof under condition {\rm (C4)}}
In contrast to the unbounded case, a trunction of the $Z_{i,j,l}$'s  is not needed to show conditions {\rm (A1)} and {\rm (A2)}.
Writing $x_n:=x/d_n+d_n$, an application of \cite[p. 251, 8. in Section VIII.4]{petrov:1972} yields 
\begin{align}\label{eq:sesegsdd.2.2}
	&\frac{p^2}{2} \P\Big( \frac{1}{\sqrt n} \sum_{l=1}^n Z_{1,2,l} > x_n \Big)\sim \frac{p^2}{2} \ov \Phi(x_n)\to \e^{-x} \,, \qquad x\in \R\,,\nto\,,
\end{align}
for $p=\exp(o(n^{1/3}))$ under {\rm (C2)}, and hence, under {\rm (C4)}. Thereby, condition (A1) holds.\\
Because the $Z_{i,j,l}$ are bounded by $K$, an application of Theorem 1.1 of \cite{Zaitsev:1987} gives
\begin{align}
	&\P \Big(\min_{j\in \{2,3\}} \frac{1}{\sqrt n} \sum_{l=1}^n Z_{1,j,l} >x_n\Big)\leq \P(\min_{j\in \{1,2\}} N_j>x_n-\lambda_n)+ c_1\exp\Big(-c_2\frac{\sqrt{n}\lambda_n}{K}\Big),\label{eq:b2}
\end{align}
where $\lambda_n=(\log p)^{-1/2}$ and $N_1,N_2$ are standard normal variables with $\Cov(N_1,N_2)=\rho$.
For $p= \exp(o( n^{\frac{1}{3}}))$, the exponential term 
$p^3\exp\big(-c_2\sqrt{n}\lambda_n/K\big)$ converges to zero.
Finally, we see that $p^3\, \P\big(\min_{j\in \{1,2\}} N_j>x_n-\lambda_n \big)\to 0$ by Lemma \ref{lem:minN}, finishing the proof of (A2) under condition (C4).

\subsection{An auxiliary result}
The following lemma is needed in the proof of Theorem \ref{thm:summax}.
\begin{lemma}\label{lem:minN}
	Let $Z,Z_1,Z_2$ be identically distributed random variables with $\E[Z]=0$, $\Var(Z)=1$ and $\Cov(Z_1,Z_2)=\rho\leq\frac{1}{3}$, which satisfy one of the conditions {\rm (C1)}, {\rm (C3)} and {\rm (C4)}. Let
	$\hat{Z}_{j}:=Z_{j}\mathds{1}_{\lbrace|Z_{j}|\leq \tau_n\rbrace}-\E[Z\mathds{1}_{\lbrace|Z|\leq \tau_n\rbrace}]$ for $j\in\{1,2\}$, where 
	\begin{itemize}
		\item
		$\tau_n=n^{1/s}$, if {\rm(C1)} holds;
		\item
		$\tau_n=\Big(\frac{2}{\eta}\log (pn)\Big)^{1/r}$,  if {\rm(C3)} holds;
		\item
		$\tau_n=K$, if {\rm (C4)} holds;
	\end{itemize}
	and set
	\begin{align*}
		\rho_n&:=\Cov(\hat{Z}_{1},\hat{Z}_{2})=\rho-\E[Z_{1}Z_{2}\mathds{1}_{\lbrace \max_{j\in\{1,2\}}|Z_{j}|> \tau_n\}}]-\Big(\E[Z_{1}\mathds{1}_{\lbrace|Z_{1}|\leq \tau_n\rbrace}]\Big)^2.
	\end{align*}
	(Under {\rm (C4)} we have $\hat{Z}_j=Z_j$ and $\rho_n=\rho$.) For a normal distributed vector
	\begin{equation*}
		\begin{pmatrix} N_1 \\ N_2 \end{pmatrix}
		\sim \mathcal{N} \Big( \begin{pmatrix} 0 \\ 0 \end{pmatrix}, 
		\begin{pmatrix}
			\Var(\hat{Z_1}) & \rho_n\\
			\rho_n & \Var(\hat{Z_1})
		\end{pmatrix} \Big),
	\end{equation*}
	and 
	\begin{itemize}
		\item
		$p=O(\frac{s-2}{4})$, if {\rm(C1)} is valid,
		\item
		$p=\exp(o(n^{1/(3+2/r)}))$, if {\rm (C3)} is valid,
		\item
		$p=\exp(o(n^{1/3}))$, if {\rm(C4)} is valid,
	\end{itemize}
	it holds that
	\begin{align}\label{p3min}
		\lim_{\nto} p^3\, \P\Big(\min_{j\in \{1,2\}} N_j>x/d_n+d_n-\lambda_n \Big)&=0\,, \qquad x\in \R\,,
	\end{align}
	where $\lambda_n:=1/\sqrt{\log p}$ and $d_n$ is defined as in Theorem~\ref{thm:summax}.
\end{lemma}

\begin{proof}
	According to Lemma \ref{lem:2} it suffices to show \eqref{p3min} for $\rho_n\geq 0$. 
	Writing $x_n:=d_n+\frac{x}{d_n}$ for $x\in \R$, an application of Lemma~\ref{lem:1} gives
	\begin{align}
		\P\bigg(\min_{j\in \{1,2\}} N_j > x_n-\lambda_n \bigg)&=\P\bigg(\min_{j\in \{1,2\}} \frac{N_j}{\sqrt{\Var(N_1)}} >\frac{x_n-\lambda_n}{\sqrt{\Var{N_1}}} \bigg)\\
		&\sim \frac{(1+\rho_n)^{3/2}\Var(N_1)}{2\pi (1-\rho_n)^{1/2} (x_n-\lambda_n)^2} \exp\bigg(\frac{-(x_n-\lambda_n)^2}{(1+\rho_n)\Var(N_1)}\bigg)\\
		&\lesssim \frac{1}{ (x_n-\lambda_n)^2} \exp\bigg(\frac{-(x_n-\lambda_n)^2}{1+\rho_n}\bigg), \qquad n\to\infty.
	\end{align}
	By the definition of $d_n$, we have
	\begin{align*}
		(x_n-\lambda_n)^2=x_n^2+O(1)=d_n^2+O(1)= 4\log p-\log(2\log p)+O(1),\qquad n\to\infty.
	\end{align*}
	Hence, we conclude that
	\begin{align}
		\P\bigg(\min_{j\in \{1,2\}} N_j > x_n-\lambda_n \bigg)&\lesssim \frac{1}{ \log p}\exp\bigg(-\frac{4\log p - \log(2\log p) }{1+\rho_n}\bigg) \notag\\
		&\lesssim (\log p)^{-\frac{\rho_n}{1+\rho_n}} p^{-4/(1+\rho_n)} \,, \qquad \nto\,, \label{minnj}
	\end{align}
	which implies \eqref{p3min} provided that
	\begin{equation}\label{520}
		p^{-\beta_n}(\log p)^{-\frac{\rho_n}{1+\rho_n}} \to 0\,,\qquad n\to\infty\,,
	\end{equation}
	where $\beta_n=(1-3\rho_n)/(1+\rho_n)$.
	
	Obviously, the claim is true for $\beta_n\geq 0$. If $\beta_n<0$, we must have $\rho=1/3$ because from $\rho< 1/3$ it follows that $\rho_n<1/3$ for $n$ large enough, since $\rho_n$ converges to $\rho$. If  $\beta_n<0$, then $\rho_n>1/3$, and thus, $\E[Z_1Z_2\1_{\{\max_{j\in\{1,2\}}|Z_j|>\tau_n\}}]<0$. Therefore, we have
	\begin{align}
		|\beta_n|=\frac{|1-3\rho_n|}{1+\rho_n}&= \Big|\frac{3}{1+\rho_n}\big[\E[Z_{1}Z_{2}\mathds{1}_{\lbrace \max_{j\in\{1,2\}}|Z_{j}|> \tau_n\}}]+\big(\E[Z\mathds{1}_{\lbrace|Z|\leq \tau_n\rbrace}]\big)^2\big]\Big|\nonumber\\
		&\leq \Big|\frac{3}{1+\rho_n}\E[Z_{1}Z_{2}\mathds{1}_{\lbrace \max_{j\in\{1,2\}}|Z_{j}|> \tau_n\}}]\Big|.\label{beta}
	\end{align}
	
	Under {\rm (C1)}, \eqref{beta} is up to a positive constant bounded above by
	\begin{align*}
		&\E[ Z^2\mathds{1}_{\lbrace |Z|> \tau_n\}}]
		\leq \frac{\E[|Z|^s]}{\tau_n^{s-2}}\lesssim \tau_n^{-(s-2)}\,.
	\end{align*}
	For $p=O(n^{(s-2)/4})$ and $\tau_n=n^{1/s}$, we therefore have
	\begin{align*}
		p^{-\beta_n}\leq p^{C/n^{(s-2)/s}}= \exp\Big(\frac{C}{n^{(s-2)/s}} \,\log p\Big),
	\end{align*}
	which tends to $1$ as $n\to\infty$ since $s>2$.
	
	Under {\rm (C3)}, \eqref{beta} is up to a positive constant bounded above by
	\begin{align*}
		&\E[ Z^2\mathds{1}_{\lbrace |Z|> \tau_n\}}]\leq \frac{\E[|Z|^2(\exp(\eta|Z|^r ))^{1/4}]}{(\exp(\eta\tau_n^r))^{1/4}} \lesssim (\exp(\eta\tau_n^r))^{-1/4}.
	\end{align*}
	For $p=\exp(o(n^{(3+\frac{2}{r})^{-1}}))$ and $\tau_n=\Big(\frac{2}{\eta}\log (pn)\Big)^{1/r}$, we deduce that $p^{-\beta_n}\leq p^{C/(pn)^{1/2}}\to 1$.
	
	Finally, under {\rm(C4)}, one has $\beta_n=0$.
	This establishes \eqref{520} in all cases and completes the proof of the lemma.
\end{proof}

\section{Proofs of the remaining results}\label{sec:otherproofs}\setcounter{equation}{0}
\subsection{Proof of Theorem \ref{thm:interpoint}}
Set $Z_{i,j,l}=\frac{(X_{il}-X_{jl})^2-2}{\sqrt{2(\E[X^4]+1)}}$ for $1\leq l\leq n$ and $1\le i<j\le p$. Then we have
\beam
\frac{D_{ij}^{(2)}-2n}{\sqrt{2n(\E[X^4]+1)}} = \frac{1}{\sqrt{n}}\sum_{l=1}^n Z_{i,j,l}.
\eeam
By simple calculations one can check that
\beam
\Cov(Z_{i,j,1},Z_{r,s,1})=\E[Z_{i,j,1}Z_{r,s,1}]=\frac{\E[X^4]-1}{2(\E[X^4]+1)}\quad \text{for}\,\, |\{i,j,r,s\}|=3\,.
\eeam
Moreover, we observe that $\frac{\E[X^4]-1}{2(\E[X^4]+1)}\leq 1/3$ if $\E[X^4]\leq 5$. If $\E[X^6]<\infty$, one obtains for the third moment 
\begin{align*}
	\E[Z_{1,2,1}^3]=\frac{\E[X^6]+9\E[X^4]-10(\E[X^3])^2-10}{\sqrt{2}(\E[X^4]+1)^{3/2}}=\tilde\kappa\,.
\end{align*}  
Additionally, the conditions {\rm (C1)}--{\rm (C4)} for Theorem \ref{thm:summax} 
follow from {\rm (B1)}--{\rm (B4)}.
Recalling the notations $b_n^{(2)}$ and $c_n^{(2)}$ in Theorem~\ref{thm:interpoint}, 
we see that
\begin{equation*}
	\max_{1\le i <j\le p} c_n^{(2)} (D_{ij}^{(2)}-b_n^{(2)})=
	\max_{1\le i <j\le p} d_n \Big(\frac{1}{\sqrt{n}}\sum_{l=1}^n Z_{i,j,l}-d_n\Big)\,.
\end{equation*}
Finally, an application of Theorem \ref{thm:summax} 
establishes the claim of Theorem \ref{thm:interpoint}.

\subsection{Proof of Theorem \ref{cor:minimaxi} }
For $x,y\in\R$ we set $Q_{x,y}:=Q^{(n)}_{x,y}:=(d_n+x/d_n,\infty)\cup (-\infty,-d_n+y/d_n]$ and recall the notation $T_{ij}=n^{-1/2}\sum_{l=1}^n Z_{i,j,l}$. Additonally, we write $M_{(1)}:=\max_{1\leq i<j\leq p}T_{ij}$ and $M_{(\tilde{p})}:=\min_{1\leq i<j\leq p}T_{ij}$.
We have
\begin{align*}
	&\P(d_n(M_{(1)}-d_n)\leq x, d_n(M_{(\tilde{p})}+d_n\leq y)
	=\P(d_n(M_{(1)}-d_n)\leq x)-G_n(x,y)\,,
\end{align*}
where
\begin{align*}
	G_n(x,y)&:=\P(d_n(M_{(1)}-d_n)\leq x, d_n(M_{(\tilde{p})}+d_n>y) = \P\Big(\sum_{1\leq i<j\leq p}\mathds{1}_{Q_{x,y}}(T_{ij})=0\Big).
\end{align*}
It suffices to show that
\begin{align}\label{eq:dsgfsdd}
	\lim_{\nto} G_n(x,y)=\exp(-(\e^{-x}+\e^y))=\P(G\leq x)\P( G\leq -y),
\end{align}
which would imply
\begin{align*}
	\lim_{\nto} \P(d_n(M_{(1)}-d_n)\leq x, d_n(M_{(\tilde{p})}+d_n\leq y)
	&= \P(G\leq x)-\P(G\leq x)\P( G\leq -y)\\
	&=\P(G\leq x)\,\P( -G\leq y)\,.
\end{align*}
As in the proof of Theorem \ref{thm:summax}, 
we may replace the $Z_{i,j,l}$ by their truncated versions $\hat{Z}_{i,j,l}$ without changing the limit of $G_n(x,y)$. For simplicity we will from now on assume that $G_n(x,y)$ and $T_{ij}$ are defined as above but with $\hat{Z}_{i,j,l}$ instead $Z_{i,j,l}$.
As in Section \ref{prel} of the proof of Theorem \ref{thm:summax}, 
equation \eqref{eq:dsgfsdd} follows from
\begin{itemize}
	\item[(A1')] $\tilde{p}\, \P( T_{12}\in Q_{x,y}) \to \e^{-x}+\e^y$ and
	\item[(A2')] $\P(T_{12}\in Q_{x,y},T_{13}\in Q_{x,y})=o(p^{-3})$,
\end{itemize}
where $x,y\in \R$ and $\nto$.

We first consider assertion (A1'). For sufficiently large $n$ we have
\begin{align*}
	\P( T_{12}\in Q_{x,y})&=\P(T_{12}> d_n+x/d_n \,\,\text{or}\, -T_{12}\geq d_n-y/d_n)\\
	&=\P(T_{12}> d_n+x/d_n)+\P(-T_{12}\geq d_n-y/d_n)\,.
\end{align*}
Therefore, (A1') follows from (A1) in the proof of Theorem \ref{thm:summax}.\\
For (A2') and sufficiently large $n$ we get
\begin{align*}
	&\P(T_{12}\in Q_{x,y},T_{13}\in Q_{x,y})\\
	&=\P((T_{12}> d_n+x/d_n\,\, \text{or}\, -T_{12}\geq d_n-y/d_n),(T_{13}> d_n+x/d_n\,\, \text{or}\,-T_{13}\geq d_n-y/d_n))\\
	&=\P(T_{12}> d_n+x/d_n, T_{13}> d_n+x/d_n)+ 2\P(T_{12}> d_n+x/d_n, -T_{13}\geq d_n-y/d_n )\\
	&\quad +\P(-T_{12}\geq d_n-y/d_n, -T_{13}\geq d_n-y/d_n)
\end{align*}
The fact that the first and the last terms are $o(p^{-3})$ follows directly from (A2) in the proof of Theorem \ref{thm:summax}. 
The argument for the middle term is similar. This establishes (A2') and finishes the proof.  

\subsection{Proof of Proposition \ref{cor:opt}}

Since $d_n\sim 2\sqrt{\log p}$, as $\nto$, part (i) of Proposition \ref{cor:opt} 
is an immediate consequence of Theorem~\ref{thm:summax}.

We turn to the proof of (ii). If 
\begin{align*}
			\max_{1\leq i<j\leq p}\frac{1}{\sqrt{\log p}}\, T_{ij}\to 2
		\end{align*}
holds, we have for any constant $C_0>2$,
\begin{align*}
	\P\Big(\max_{1\leq i<j\leq p}\frac{1}{\sqrt{n\log p}}\sum_{l=1}^nZ_{i,j,l}\geq C_0\Big)\to 0, \qquad \nto\,.
\end{align*}
In view of the inequality $\max_{1\leq i<j\leq p}\sum_{l=1}^nZ_{i,j,l}\ge \max_{1\leq i< p/2}\sum_{l=1}^nZ_{i,[p/2]+i,l}$, where $[p/2]$ denotes the integer part of $p/2$, and because $\big\{\sum_{l=1}^nZ_{i,[p/2]+i,l},\,\, 1\leq i<p/2\big\}$ are iid random variables, we have
\begin{align*}
	&1-\left(1- \P\Big(\sum_{l=1}^nZ_{1,2,l}\geq C_0\sqrt{n\log p}\Big) \right)^{[p/2]} =
	\P\Big(\max_{1\leq i< p/2}\sum_{l=1}^nZ_{i,[p/2]+i,l}\geq C_0\sqrt{n\log p}\Big)\\
	&\le \P\Big(\max_{1\le i<j\le p}\sum_{l=1}^nZ_{i,j,l}\geq C_0\sqrt{n\log p}\Big)\to 0,\qquad n\to\infty.
\end{align*} 
We deduce that $p\,\P\big(\sum_{l=1}^nZ_{1,2,l}\geq C_0\sqrt{n\log p}\big)\to 0$ and since the same arguments hold for $-Z$, we obtain
\begin{align*}
	p\,\P\Big(\big|\sum_{l=1}^nZ_{1,2,l}\big|\geq C_0\sqrt{n\log p}\Big)\to 0,\qquad n\to\infty.
\end{align*}
Now, we set $D_n:=\{\big|\sum_{l=2}^n Z_{1,2,l}\big|<\sqrt{n}\}$. By the central limit theorem, we know that $\P(D_n)$ converges to $2\Phi(1)-1\geq 1/2$ as $n\to\infty$. 
We have 
\begin{align*}
	\bigg\{\frac{\big|\sum_{l=1}^nZ_{1,2,l}\big|}{\sqrt{n}}\geq C_0\sqrt{\log p} \bigg\} \cap D_n \supseteq 
	\bigg\{\frac{|Z_{1,2,l}| - \sqrt{n}}{\sqrt{n}} \geq C_0\sqrt{\log p} \bigg\} \cap D_n\,,
\end{align*}
where the two events on the \rhs\ are independent. 
Thus, we get for sufficiently large $n$
\begin{align}
	p\,\P\Big(\Big|\sum_{l=1}^nZ_{1,2,l}\Big|\geq C_0\sqrt{n\log p}\Big)&\ge p\,\P\Big(\Big|\sum_{l=1}^nZ_{1,2,l}\Big|\geq C_0\sqrt{n\log p}, D_n\Big)\nonumber\\
	&\geq p\,\P(D_n)\P(|Z|-\sqrt{n}\geq C_0\sqrt{n\log p})\nonumber\\
	&\geq \frac{p}{2}\P(|Z|\geq (C_0+1)\sqrt{n\log p})\to 0,\qquad n\to\infty\label{rueck}
\end{align}
for any $p$ with $\log p=o(n^s)$.
Following the argument in \cite[p.632]{shao2014necessary}, this tail decay of $|Z|$ can be used to show that $\E\big[\exp\big(\eta|Z|^{\frac{2s}{1+s}}\big)\big]<\infty$ for some $\eta>0$. We omit details. 

\subsection{Proof of Proposition \ref{thm:rate}}

As in Section \ref{prel} in the proof of Theorem \ref{thm:summax}, 
we introduce $I=I^{(n)}= \{(i,j): 1\le i<j\le p\}$ and for $\alpha=(i,j) \in I$, we set $\eta_{\alpha}=  d_n\big(\frac{1}{\sqrt{n}}\sum_{l=1}^n Z_{i,j,l}-d_n\big)$. In addition, we define
\begin{equation*}
	B_{\alpha}= \{(k,l)\in I \,:\, |\{i,j,k,l\}| \leq 3\}\,,
\end{equation*}
and for $x\in \R$ we write
\begin{equation*}
	\begin{split}
		\lambda &:=\lambda^{(n)}:=\sum_{\alpha\in I} \P( \eta_{\alpha} > x). 
	\end{split}
\end{equation*}
Then we get by Lemma \ref{lem:poissonapprox} and the triangle inequality
\begin{align}
	&\Big|\P\Big(\max_{1\leq i<j\leq p}d_n\Big(\frac{1}{\sqrt{n}}\sum_{l=1}^n Z_{i,j,l}-d_n\Big)\leq x\Big)-\e^{-e^{-x}}\Big|\\
	&\leq (1+\lambda^{-1})(b_1+b_2+b_3)+|\e^{-\lambda}-\e^{-e^{-x}}|,\label{ineq:rate}
\end{align}
where
\begin{equation*}
	\begin{split}
		b_1&= \sum_{\alpha \in I} \sum_{\beta \in B_{\alpha}} \P(\eta_{\alpha} >x) \P(\eta_{\beta} >x)\,, \\
		b_2&= \sum_{\alpha \in I} \sum_{\beta \in B_{\alpha}\backslash \{\alpha\}} \P(\eta_{\alpha} >x, \eta_{\beta} >x)\,,\\
		b_3&= \sum_{\alpha \in I} \E\Big| \P(\eta_{\alpha}>x |\sigma(\eta_{\beta}, \beta \notin B_{\alpha})) -  \P(\eta_{\alpha} >x) \Big|\,.
	\end{split}
\end{equation*}

We already know that $b_3=0$ and, by \eqref{eq:b1}, $b_1=O(p^{-1})$. From \eqref{eq:b2} we get
\begin{align*}
	&b_2\leq p^3 \P(\min(N_1,N_2) >x_n-(\log p)^{-1/2})+ p^3 c_1\exp\Big(-c_2\sqrt{n(\log p)^{-1}}\Big),
\end{align*}
where $x_n=x/d_{n,1}+d_{n,1}$ and $N_1,N_2$ are standard normal variables with $\Cov(N_1,N_2)=\rho$. Since $p=\exp(o(n^{1/3}))$ we may write $p=\exp(\gamma_n n^{1/3})$, where $\gamma_n\to 0$ as $n\to \infty$. Hence, we obtain for the second term
\begin{align*}
	p^3 \exp\Big(-c_2\sqrt{n(\log p)^{-1}}\Big)&=\exp(3\gamma_n n^{1/3}-c_2n^{1/3}\gamma_n^{-1/2})\\
	&=\exp(\gamma_n n^{1/3}(3-c_2\gamma_n^{-3/2})= O( p^{-1}).
\end{align*}
For the first term we get by \eqref{minnj}
\begin{align*}
	p^3 \P(\min(N_1,N_2)>x_n-(\log p)^{-1/2})&\lesssim p^3\, (\log p)^{-\frac{\rho}{1+\rho}} p^{-4/(1+\rho)}\\
	&=O((\log p)^{-\rho/(1+\rho)}p^{-(1-3\rho)/(1+\rho)})\,, \qquad \nto\,.
\end{align*}
Therefore, the first term of \eqref{ineq:rate} is of the order $O((\log p)^{-\rho/(1+\rho)}p^{-(1-3\rho)/(1+\rho)})$.\\

Now, we consider the second term of \eqref{ineq:rate}. By the mean value theorem there exists a $y_n$ between $\lambda$ and $\e^{-x}$ such that
\begin{align}\label{ineq:2}
	|e^{-\lambda}-\e^{-e^{-x}}|=\e^{-y_n}|\lambda-\e^{-x}|\leq e^{-y_n}(|\lambda-\tilde{p}\,\bar{\Phi}(x_n)|+|\tilde{p}\,\bar{\Phi}(x_n)-\e^{-x}|).
\end{align}
We proceed by bounding the \rhs\ of \eqref{ineq:2}. For the first term we get by Lemma~\ref{lin}(iii)
\begin{align*}
	|\lambda-\tilde{p}\,\bar{\Phi}(x_n)|&=\tilde{p}\Big|\P\Big(\frac{1}{\sqrt{n}}\sum_{l=1}^n Z_{1,2,l}> x_n\Big)-\bar{\Phi}(x_n)\Big|\\
	&=\tilde{p}\,\bar{\Phi}(x_n)\Big|\exp\Big(\frac{x_n^3\E[Z^3]}{6\sqrt{n}}\Big)\Big[1+O\Big(\frac{1+x_n}{\sqrt{n}}\Big)\Big]-1\Big|.
\end{align*}
By the mean value theorem there exists a $\tilde{y}_n$ between $0$ and $x_n^3\E[Z^3]/(6\sqrt{n})$ with
\begin{align*}
	\Big|\exp\Big(\frac{x_n^3\E[Z^3]}{6\sqrt{n}}\Big)-1\Big|=\e^{\tilde{y}_n} \frac{x_n^3|\E[Z^3]|}{6\sqrt{n}}=O\Big(\sqrt{\frac{(\log p)^3}{n}}\Big)\,, \qquad \nto\,.
\end{align*} 
As $\tilde{p}\,\bar{\Phi}(x_n)\to \e^{-x}$ and $(1+x_n)n^{-1/2}=O(\sqrt{n^{-1}\log p})$, we get\\ $|\lambda-\tilde{p}\,\bar{\Phi}(x_n)| =O(\sqrt{n^{-1}(\log p)^3})$.
For the second term of \eqref{ineq:2} we have
\begin{align*}
	|\tilde{p}\,\bar{\Phi}(x_n)-\e^{-x}|\leq \tilde{p}\,\Big|\bar{\Phi}(x_n)-\frac{\phi(x_n)}{x_n}\Big|+\Big|\tilde{p}\,\frac{\phi(x_n)}{x_n}-\e^{-x}\Big|\,,
\end{align*}
where $\phi$ is the density of the standard normal distribution. 
Using the following classical inequality (e.g.~\cite{feller:1968}) for the tail of the standard normal distribution function 
\begin{equation*}
	\Big(\frac{1}{y}-\frac{1}{y^3}\big) \, \phi(y) \le \bar\Phi(y)\le \frac{1}{y} \,\phi(y)\,, \qquad y>0\,,
\end{equation*}
we deduce that
\begin{align*}
	\tilde{p}\,\Big|\bar{\Phi}(x_n)-\frac{\phi(x_n)}{x_n}\Big|\leq \tilde{p}\frac{\phi(x_n)}{x_n^3}=O((\log p)^{-1}).
\end{align*}
Additionally, we get by the definition of $d_{n,1}$ 
\begin{align*}
	\Big|\tilde{p}\,\frac{\phi(x_n)}{x_n}-\e^{-x}\Big|&\lesssim \bigg|\frac{\e^{-\frac{1}{2}(x^2/d_{n,1}^2+c_n)}\sqrt{\log \tilde{p}}}{x/d_{n,1}+d_{n,1}}-1\bigg|\lesssim \Big|\e^{-\frac{1}{2}(x^2/d_{n,1}^2+c_n)}-1\Big|,
\end{align*}
where $c_n:= (\log\log\tilde{p}+\log 4\pi)^{2}/(8\log\tilde{p})$.
By the mean value theorem there exists a $\hat{y}_n$ between $0$ and $-1/2(x^2/d_{n,1}^2+c_n)$ with
\begin{align*}
	\Big|\e^{-\frac{1}{2}(x^2/d_{n,1}^2+c_n)}-1\Big|=\Big(x^2/d_{n,1}^2+c_n\Big)\e^{\hat{y}_n}=O\Big(\frac{(\log(\log p))^{2}}{\log p}\Big).
\end{align*}
To summarize, the second term of \eqref{ineq:rate} is of order {\small $O(\sqrt{n^{-1}(\log p)^3})+O((\log p)^{-1}\log\log p)$}, which finishes the proof.

\subsection{Proof of Proposition \ref{prop:momentsv}}

We check that for $\rho\in [0,1/2]$, $\tilde{p}=p(p-1)/2$ and $x_n:=x/d_{n,1}+d_{n,1}$
\begin{align*}
	\E[V_{n,\rho}(x)]=\tilde{p}\, \P(Y_{12}>x_n)=\tilde{p}\,\bar{\Phi}(x_n) \to \e^{-x}
\end{align*} 
by the choice of $d_{n,1}$.
For the second moment we get
\begin{align*}
	\E[V_{n,\rho}^2(x)]&=\sum_{1\leq i_1<j_1\leq p}\sum_{1\leq i_2<j_2\leq p}\P(Y_{i_1j_1}>x_n, Y_{i_2j_2}>x_n)   \\
	&=\tilde{p}\,\bar{\Phi}(x_n)+ \tilde{p}\,\frac{(p-2)(p-3)}{2}\bar{\Phi}(x_n)^2+\tilde{p}\,(2p-4)\P(Y_{12}>x_n, Y_{13}>x_n)\,.
\end{align*}

The last expression is asymptotically equal to
\begin{align}
	&\tilde{p}\,\bar{\Phi}(x_n)+ \big(\tilde{p}\,\bar{\Phi}(x_n)\big)^2+p^3\P(Y_{12}>x_n, Y_{13}>x_n)\nonumber\\
	&\sim \e^{-x}+\e^{-2x}+\frac{(1+\rho)^{3/2}}{8\pi (1-\rho)^{1/2}}\, (\log p)^{-\rho/(1+\rho)}p^{(3\rho-1)/{(1+\rho)}}, \label{2mom}
\end{align}
where we applied Lemma \ref{lem:1} (see also the proof of Lemma \ref{lem:minN}) and the fact that $\tilde{p}\,\bar{\Phi}(x_n) \to \e^{-x}$ in the last step. The last term of \eqref{2mom} tends to zero for $\rho\in[0,1/3]$ and to infinity for $\rho\in(1/3,1/2]$.

\subsection{Proof of Theorem~\ref{thm:firstord}}

Let  $0\le \rho \le 1/2$.  We will provide an explicit construction of the field $(Y_{ij})$. Let $(\xi_i)_{i\ge 1}$ and $(\eta_j)_{j\ge 1}$ be two independent sequences of iid standard Gaussian random variables and let $(N_{ij})_{i,j\ge 1}$ be a field of iid standard Gaussians. It is easy to check that 
$$U_{ij}:=\sqrt{1-2\rho} N_{ij} + \sqrt{\rho}(\xi_i+\eta_j)\,,\qquad 1\le i <j\,,$$ 
are standard Gaussian random variables with the same covariance function as $(Y_{ij})$, i.e., $\Cov(Y_{ij},Y_{st})=\Cov(U_{ij},U_{st})$. Therefore we may assume that 
\begin{equation}\label{eq:defXij}
	Y_{ij}=Y_{ij}^{(\rho)}:=\sqrt{1-2\rho} N_{ij} + \sqrt{\rho}(\xi_i+\eta_j)\,,\qquad 1\le i <j\,,
\end{equation}
where superscript $(\rho)$ highlights the dependence on $\rho$. We will need the following notation
\begin{equation*}
	M_{n1}^{(\rho)}:=\max_{1\le i<j\le p} Y_{ij}^{(\rho)}  \qquad \text{ and } \qquad M_n^{(\rho)}:=\max_{1\le i\neq j\le n} Y_{ij}^{(\rho)}\,.
\end{equation*}
An important tool will be Slepian's lemma \cite{slepian:1962} (see also \cite[Corollary 4.2.3]{leadbetter2012extremes}): If $G,G'\in \R^n$ are centered Gaussian vectors with standarized entries, and if their correlation matrices $R,R'$ satisfy the entrywise inequality $R_{ij}\le R_{ij}'$ for all $1\le i,j\le p$, then it holds for all $t\in \R$ that
\begin{equation}\label{eq:slepian}
	\P\left( \max_{i=1,\ldots,p} G_i \le t \right) \le \P\left( \max_{i=1,\ldots,p} G_i' \le t \right)\,.
\end{equation}
An application of \eqref{eq:slepian} yields for $\rho \in[0,1/2]$
\begin{equation}\label{eq:applslepian}
	\P(M_{n1}^{(0)}\le t)\le \P(M_{n1}^{(\rho)}\le t) \le \P(M_{n1}^{(1/2)}\le t)\,,\qquad t\in \R\,.
\end{equation}
For $\vep>0$ it follows that
$$\P\left(\frac{M_{n1}^{(\rho)}}{\sqrt{\log p}}\ge 2+\vep\right) \le \P\left(\frac{M_{n1}^{(0)}}{\sqrt{\log p}}\ge 2+\vep\right)\,.$$
As $M_{n1}^{(0)}$ is a maximum of $\tilde{p}$ iid standard Gaussian random variables it is well-known from extreme value theory (see \cite{embrechts:kluppelberg:mikosch:1997}) that $M_{n1}^{(0)}/\sqrt{\log p}\cip 2$ from which we conclude 
\begin{equation}\label{eq:dgsdgd}
	\lim_{\nto} \P\left(\frac{M_{n1}^{(\rho)}}{\sqrt{\log p}}\ge 2+\vep\right) =0\,, \qquad \rho\in[0,1/2]\,.
\end{equation}
Applying \eqref{eq:slepian} and using the fact that $$M_n^{(1/2)}\ge \frac{1}{\sqrt{2}} \Big(\max_{1\le i \le p/2} \xi_i +  \max_{p/2+1\le j \le p} \eta_j\Big)$$ 
yields
\begin{align*}
	\P\left(\frac{M_{n1}^{(\rho)}}{\sqrt{\log p}}\le 2-\vep\right) &\le \P\left(\frac{M_{n1}^{(1/2)}}{\sqrt{\log p}}\le 2-\vep\right)\\
	&\le \P\left(\frac{ \max_{1\le i \le p/2} \xi_i +  \max_{p/2+1\le j \le p} \eta_j}{\sqrt{2} \,\sqrt{\log p}}\le 2-\vep\right)\,.
\end{align*}
Since $\max_{1\le i \le p/2} \xi_i /\sqrt{\log p} \cip \sqrt{2}$ and $\max_{p/2+1\le j \le p} \eta_j /\sqrt{\log p} \cip \sqrt{2}$ we obtain
\begin{equation*}
	\lim_{\nto} \P\left(\frac{M_{n1}^{(\rho)}}{\sqrt{\log p}}\le 2-\vep\right) =0\,, \qquad \rho\in[0,1/2]\,.
\end{equation*}
which in conjunction with \eqref{eq:dgsdgd} proves that 
\begin{align}\label{conv11}
	\max_{1\leq i<j\leq p}\frac{1}{\sqrt{\log p}}\, Y_{ij}\cip 2\,, \qquad \nto\,.
\end{align}
From \eqref{conv11} and Lemma 2.11 
it follows for all $\rho\in[0,1/2]$ that
\begin{align*}
	\max_{1\leq i<j\leq p}\frac{1}{\sqrt{\log p}}\, T_{ij}\cip 2\,, \qquad \nto\,,
\end{align*}
establishing the desired result.

\section{Maximum interpoint distance between two samples}\label{sec:twosample}

	Instead of considering the largest interpoint distance between all possible combinations of points of one sample, we can also take a look at the largest distance between points of two different samples. Let $(\x_i)_{i\le p}$ and $(\y_j)_{j\leq p}$ be two iid sequences of $\R^n$-valued random vectors, which are independent from each other and whose components fulfill the standard conditions. We study the asymptotic distribution of the maximum of the modified interpoint distances,
	\begin{align*}
	\hat{D}_{ij}^{(2)}:=\Vert\x_i-\y_j\Vert_2^2,\quad 1\leq i,j\leq p.
	\end{align*}
	The distances $\hat{D}_{ij}^{(2)}, 1\leq i,j\leq p$ are not independent and for the correlations it holds that
	\begin{align*}
		\rho_1:=\Corr(\hat D_{ij}^{(2)},\hat D_{ik}^{(2)})&=\frac{\E[X^4]-1}{\E[X^4]+ \E[Y^4]+2}, \qquad i,j,k\leq p,\,j\neq k\,,\\
		\rho_2:=\Corr(\hat D_{ij}^{(2)},\hat D_{hj}^{(2)})&=\frac{\E[Y^4]-1}{\E[X^4]+ \E[Y^4]+2}, \qquad i, j, h\leq p,\,i\neq h.
	\end{align*}
	Additionally, we define the sequence of norming constants
	\begin{align*}
		\hat{d}_{n}:=\sqrt{2\log p^2} - \dfrac{\log\log p^2+\log 4\pi}{2(2\log p^2)^{1/2}}.
	\end{align*}
	To formulate an analogous result to Theorem \ref{thm:interpoint} 
 for the case of two samples, we need similar assumptions as (B1)-(B4).
	\begin{enumerate}
		\item[\rm (B1')]
		There exists $s>2$ \st\ $\E[|X|^{2s}(\log(|X|))^{s/2}]< \infty$ and $\E[|Y|^{2s}(\log(|Y|))^{s/2}]< \infty$. Additionally, $\rho_1\leq \frac{1}{3}$ and $\rho_2\leq \frac{1}{3}$.
		\item[\rm (B2')] 
		There exist constants $\eta>0$ and $0<r\leq 2/3$ \st\ $\E[\exp(\eta\,|X|^{2r})]<\infty$ and\linebreak $\E[\exp(\eta\,|Y|^{2r})]<\infty$. Additionally, $\rho_1< \frac{1}{3}$ and $\rho_2<\frac{1}{3}$.
	\end{enumerate}
\begin{theorem}\label{thm:interpoint2}
		Let $(\x_i)_{i\le p}$ and $(\y_j)_{j\leq p}$ be two iid sequences of $\R^n$-valued random vectors, which are independent from each other and whose components fulfill the standard conditions. Assume one of the conditions {\rm (B1')} or {\rm (B2')} on $X$ and $Y$ and that $p=p_n\to\infty$ satisfies
		\begin{itemize}
			\item
			$p=O(n^{(s-2)/4})$, if {\rm (B1')} holds.
			\item
			$p=\exp(o(n^{r/(2-r)}))$, if {\rm (B2')} holds.
			
		\end{itemize}
		Then we have 
		$$\max_{1\le i, j\le p} \hat{c}_n^{(2)} (\hat{D}_{ij}^{(2)}-\hat{b}_n^{(2)}) \cid G\,,$$
		where $G$ is standard Gumbel distributed. The sequences $(\hat{b}_n^{(2)})$ and $(\hat{c}_n^{(2)})$ are given by
		\beam
		\hat{b}_n^{(2)}:=2n+\sqrt{n(\E[X^4]+\E[Y^4]+2)}\,\hat{d}_n \quad\text{and}\quad \hat{c}_n^{(2)}:=\frac{\hat{d}_n}{\sqrt{n(\E[X^4]+\E[Y^4]+2)}}.
		\eeam
\end{theorem}

\begin{proof}
    The proof is similar to the proof of Theorem \ref{thm:summax}. To apply Lemma \ref{lem:poissonapprox} we make the following definitions. Let $\hat{I}=\{(i,j): 1\leq i,j\leq p\}$ be an index set and for every $\alpha =(i,j)\in\hat{I}$ set $\hat{\eta}_\alpha=\hat{c}_n^{(2)} (\hat{D}_{ij}^{(2)}-\hat{b}_n^{(2)})$ and
    \begin{align*}
        \hat{B}_\alpha:=\big\{(k,l)\in\hat{I}: k=i\,\,\, \text{or}\,\,\, l=j\big\}.
    \end{align*}
    Additionally, we set $\hat{\lambda}:=\sum_{\alpha\in \hat{I}} \P( \hat\eta_{\alpha} > x)$. Then, it follows by Lemma \ref{lem:poissonapprox} that
    \begin{align*}
        \max_{1\le i, j\le p} \hat{c}_n^{(2)} (\hat{D}_{ij}^{(2)}-\hat{b}_n^{(2)}) \cid G
    \end{align*}
    if the claims
    \begin{itemize}
        \item [(A1')] $p^2\P\Big( \hat{c}_n^{(2)} \big(\hat{D}_{11}^{(2)}-\hat{b}_n^{(2)}\big)>x\Big)\to \e^{-x}$,
        \item[(A2')] $\P\Big( \hat{c}_n^{(2)} \big(\hat{D}_{11}^{(2)}-\hat{b}_n^{(2)}\big)>x,  \hat{c}_n^{(2)} (\hat{D}_{12}^{(2)}-\hat{b}_n^{(2)}\big)>x\Big)=o(p^{-3})$ and
        \item[(A3')] $\P\Big( \hat{c}_n^{(2)} \big(\hat{D}_{11}^{(2)}-\hat{b}_n^{(2)}\big)>x,  \hat{c}_n^{(2)} \big(\hat{D}_{21}^{(2)}-\hat{b}_n^{(2)}\big)>x\Big)=o(p^{-3})$,
    \end{itemize}
    where $x\in\R$ and $n\to\infty$, are fulfilled.
    Following the lines of the proof of (A1) in the proof of Theorem~\ref{thm:summax} under condition (C1) and (C2) we can show Assertion (A1'). Likewise we can derive (A2') and (A3') by following the lines of the proof of (A2). 
\end{proof}

\appendix
\section{Technical Tools}\setcounter{equation}{0}
\label{sec:appendix}
\subsection{Poisson approximation}
The first tool is a Poisson approximation, which can be found in Theorem 1 of \cite{arratia:goldstein:gordon:1989}.
\begin{lemma}\label{lem:poissonapprox2}
	Let $I$ be an index set and $\{B_{\alpha}, \alpha \in I \}$ be a set of subsets of $I$, that is, $B_{\alpha}\subset I$. Let also $\{ \theta_{\alpha}, \alpha\in I \}$ be Bernoulli random variables. Set $W=\sum_{\alpha\in I}\theta_\alpha$ and $\lambda =\E[W]= \sum_{\alpha \in I} \P(\theta_{\alpha} =1)\in (0,\infty)$. Then
	\begin{equation*}
		\Big | \P \Big( W=0 \Big) - \e^{-\lambda} \Big| \le (1\wedge\lambda^{-1}) (b_1+b_2+b_3)\,,
	\end{equation*}
	where
	\begin{equation*}
		\begin{split}
			b_1&= \sum_{\alpha \in I} \sum_{\beta \in B_{\alpha}} \P(\theta_\alpha=1) \P(\theta_{\beta} =1)\,, \\
			b_2&= \sum_{\alpha \in I} \sum_{\beta \in B_{\alpha}\backslash \{\alpha\}} \P(\theta_{\alpha} =1, \theta_{\beta} =1)\,,\\
			b_3&= \sum_{\alpha \in I} \E\Big| \P(\theta_{\alpha}=1 |\sigma(\theta_{\beta}, \beta \notin B_{\alpha})) -  \P(\theta_{\alpha} =1) \Big|\,,
		\end{split}
	\end{equation*}
	and $\sigma(\theta_{\beta}, \beta \notin B_{\alpha})$ is the $\sigma$-algebra generated by $\{\theta_{\beta}, \beta \notin B_{\alpha}\}$. In particular, if $\theta_{\alpha}$ is independent of $\{\theta_{\beta}, \beta \notin B_{\alpha}\}$ for each $\alpha$, then $b_3=0$.
\end{lemma}

The next result is a special case of Lemma \ref{lem:poissonapprox2} which is obtained by setting $\theta_{\alpha}=\1_{\{\eta_{\alpha} >t\}}$.
\begin{lemma}\label{lem:poissonapprox}
	Let $I$ be an index set and $\{B_{\alpha}, \alpha \in I \}$ be a set of subsets of $I$, that is, $B_{\alpha}\subset I$. Let also $\{ \eta_{\alpha}, \alpha\in I \}$ be random variables. For a given $t\in \R$, set $\lambda = \sum_{\alpha \in I} \P(\eta_{\alpha} >t)$. Then
	\begin{equation*}
		\Big | \P \Big( \max_{\alpha \in I} \eta_{\alpha}\le t \Big) - \e^{-\lambda} \Big| \le (1\wedge\lambda^{-1}) (b_1+b_2+b_3)\,,
	\end{equation*}
	where
	\begin{equation*}
		\begin{split}
			b_1&= \sum_{\alpha \in I} \sum_{\beta \in B_{\alpha}} \P(\eta_{\alpha} >t) \P(\eta_{\beta} >t)\,, \\
			b_2&= \sum_{\alpha \in I} \sum_{\beta \in B_{\alpha}\backslash \{\alpha\}} \P(\eta_{\alpha} >t, \eta_{\beta} >t)\,,\\
			b_3&= \sum_{\alpha \in I} \E\Big| \P(\eta_{\alpha}>t |\sigma(\eta_{\beta}, \beta \notin B_{\alpha})) -  \P(\eta_{\alpha} >t) \Big|\,,
		\end{split}
	\end{equation*}
	and $\sigma(\eta_{\beta}, \beta \notin B_{\alpha})$ is the $\sigma$-algebra generated by $\{\eta_{\beta}, \beta \notin B_{\alpha}\}$. In particular, if $\eta_{\alpha}$ is independent of $\{\eta_{\beta}, \beta \notin B_{\alpha}\}$ for each $\alpha$, then $b_3=0$.
\end{lemma}

\subsection{Large deviations and tails of multivariate Gaussian distribution}

The next lemma is due to \cite{petrov:1972} and \cite{linnik1961probability}. For the current formulation see also \cite[Lemma~3.2]{shao2014necessary}.
\begin{lemma}\label{lin}
	Let $X,X_1,X_2,\ldots$ be i.i.d. random variables with $\E[X]=0$ and $\E[X^2]=1$ and set $S_n:=\sum_{i=1}^n X_i$.
	\begin{enumerate}
		\renewcommand{\labelenumi}{(\roman{enumi})}
		\item If $\E[\exp(\eta|X|^r)]<\infty$ for some $0<r\leq 1$ and $\eta>0$, then
		\begin{align*}
			\lim\limits_{n\to\infty}\frac{1}{x_n^2}\log\P(S_n/\sqrt{n}\geq x_n)=-1/2
		\end{align*}
		for any $x_n\to\infty$, $x_n=o(n^{\frac{r}{2(2-r)}})$.
		\item  If $\E[\exp(\eta|X|^r)]<\infty$ for some $0<r\leq 1/2$ and $\eta>0$, then
		\begin{align*}
			\frac{\P(S_n/\sqrt{n}\geq x_n)}{1-\Phi(x_n)}\to 1,\qquad n\to\infty
		\end{align*}
		holds uniformly for $x_n= o(n^{\frac{r}{2(2-r)}})$, $x_n\ge 0$.
		\item  If $\E[\exp(\eta X)]<\infty$ for some $\eta>0$, then
		\begin{align*}
			\frac{\P(S_n/\sqrt{n}\geq x_n)}{1-\Phi(x_n)}=\exp\left(\frac{x_n^3\E[X^3]}{6n^{1/2}}\right)\left(1+O\left(\frac{1+x_n}{n^{1/2}}\right)\right)
		\end{align*}
		holds for $ x_n= o(n^{1/4})$, $x_n\geq 0$.
	\end{enumerate}
\end{lemma}

Next, we study the tails of a multivariate Gaussian distribution. The following lemma is a direct consequence of Slepian's Lemma (see \cite[Corollary 4.2.3.]{leadbetter2012extremes}).

\ble \label{lem:2}
Let $X_1,...,X_d$ and $X'_1,...,X'_2$ be standard Gaussian random variables with
$\operatorname{Cov}(X'_i,X'_j)\leq\operatorname{Cov}(X_i,X_j)$ for each $i,j$. Then, for all $x\in \R$
\begin{align*}
	\P\Big( \min\limits_{i=1,...,d}X'_i > x\Big)\leq \P\Big( \min\limits_{i=1,...,d}X_i > x\Big).
\end{align*}
\ele


The case of equicorrelation plays a special role in the analysis of Gaussian tails as it contains the strongest possible dependencies given an upper bound $\rho$ on the correlations of the $X_i$. 
\begin{lemma} \label{lem:1}
	Let $(X_1,...,X_d)^{\top}$ be a centered Gaussian random vector with covariance matrix 
	\begin{align*}
		\Sigma_d:= (1-\rho)\bfI_d+\rho \mathbf{11^\top},
	\end{align*}
	where $\mathbf{1}=(1,...,1)^\top \in\R^d$ and $-1/(d-1)< \rho <1$. Then, as $t\to\infty$
	\begin{align}
		\P\bigg(\min\limits_{i=1,\ldots,d}X_i >t \bigg)\sim\frac{1}{(2\pi)^{d/2}|\Sigma_d|^{1/2}}\bigg(\frac{1+(d-1)\rho}{t}\bigg)^d\exp\bigg(\frac{-\frac{1}{2}t^2d}{1+(d-1)\rho}\bigg),
	\end{align}
	where $|\Sigma_d|=(-1)^{d-1}(\rho-1)^{d-1}((d-1)\rho+1)$.
	In particular, if $d=2$ we have
	\begin{align}
		\P\bigg(\min_{i=1,2} X_i >t \bigg)&\sim \frac{(1+\rho)^{3/2}}{2\pi (1-\rho)^{1/2} t^2} \exp\bigg(\frac{-t^2}{1+\rho}\bigg)\,.
	\end{align}
\end{lemma}

\begin{proof}
	Let $\langle \cdot,\cdot \rangle$ denote the Euclidean inner product. By Example 4 of \cite{MR2007795} we get for $t\in\R$ as $t\to\infty$
	\begin{align*}
		\P\bigg(\min\limits_{i=1,...,d}X_i>t\bigg)&\sim\frac{1}{(2\pi)^{d/2}\vert\Sigma_d\vert^{1/2}} \bigg(\frac{1+(d-1)\rho}{t}\bigg)^d \exp\bigg(\frac{-\langle t\mathbf{1},\Sigma_d^{-1}t\mathbf{1}\rangle}{2}\bigg)\\
		&=\frac{1}{(2\pi)^{d/2}\vert\Sigma_d\vert^{1/2}}\bigg(\frac{1+(d-1)\rho}{t}\bigg)^d\exp\bigg(-\frac{t^2\sum_{i,j=1}^d\Sigma_{dij}^{-1}}{2}\bigg).
	\end{align*}
	For the inverse of $\Sigma_d$ we get
	\begin{align*}
		\Sigma_d^{-1}=\frac{1}{1-\rho}\bfI_d-\frac{\rho}{(1-\rho)(1+(d-1)\rho)}\mathbf{11^\top},
	\end{align*}
	and therefore, the sum of the entries of  $\Sigma_d^{-1}$ equals
	\begin{align*}
		\sum_{i,j=1}^d\Sigma_{dij}^{-1}=\frac{d}{1-\rho}-\frac{d^2\rho}{(1-\rho)(1+(d-1)\rho)}=\frac{d}{1+(d-1)\rho}\,,
	\end{align*}
	which establishes the desired result.
\end{proof}

\bibliography{libraryjohannes}
\end{document}